\documentclass[11pt]{article}

\usepackage[a4paper,margin=1in]{geometry}
\usepackage{amsmath,amssymb,amsthm,mathtools}
\usepackage{array}
\usepackage{microtype}
\usepackage{hyperref}

\hypersetup{
  colorlinks=true,
  linkcolor=blue,
  citecolor=blue,
  urlcolor=blue
}

\newtheorem{theorem}{Theorem}[section]
\newtheorem{proposition}{Proposition}[section]
\newtheorem{lemma}{Lemma}[section]
\newtheorem{corollary}{Corollary}[section]
\newtheorem{definition}{Definition}[section]
\newtheorem{remark}{Remark}[section]
\newtheorem{example}{Example}[section]
\newtheorem*{mainthm}{Main theorem}

\newcommand{\T}{\mathbb T}
\newcommand{\R}{\mathbb R}
\newcommand{\Id}{I}
\newcommand{\Spp}{\mathbb S_{++}}
\newcommand{\Sym}{\mathbb S}
\newcommand{\tr}{\operatorname{tr}}
\newcommand{\divv}{\operatorname{div}}
\newcommand{\curl}{\operatorname{curl}}
\newcommand{\Log}{\operatorname{Log}}
\newcommand{\dd}{\,dx}
\newcommand{\calE}{\mathcal E}
\newcommand{\calC}{\mathcal C}

\newcommand{\calR}{\mathcal R}
\newcommand{\calG}{\mathcal G}

\newcommand{\calD}{\mathcal D}
\newcommand{\calF}{\mathcal F}

\newcommand{\eps}{\varepsilon}
\newcommand{\norm}[1]{\left\lVert #1\right\rVert}
\newcommand{\Besov}{B^0_{\infty,1}}

\title{Pressure Quotients and Endpoint Velocity-Clock Criteria for Non-Diffusive Viscoelastic Flows}
\author{Sai Peng\\
School of Mathematics and Computational Science, Xiangtan University\\
\texttt{pscfd@xtu.edu.cn}}
\date{July 2026}

\begin{document}
\maketitle

\begin{abstract}
We prove endpoint continuation criteria for stress-diffusion-free
incompressible viscoelastic flows by working modulo pressure.  In two space
dimensions, every smooth spectral isotropic stress has the pressure-free active
form \(q_1(a,|Y|^2)Y\), where \(C=aI+Y\) and \(\tr Y=0\).  The trace-free
conformation equation contains the universal stretching block \(2aS(u)\); a
weighted active-deviatoric energy cancels the top-order coupling between this
block and \(\operatorname{div}(q_1Y)\).  On compact conformation windows this gives
a high-order estimate with coefficient
\[
  1+\|\nabla u\|_{B^0_{\infty,1}}+\|\Log C\|_{H^{1+\varepsilon}}^2 .
\]
The abstract coefficient class is only a compact-window quotient template.  The
model consequences are: strong two-dimensional Oldroyd--B solutions continue
under \(\nabla u\in L^1_tB^0_{\infty,1}\), while strong two-dimensional FENE-P
solutions continue under \(\nabla u\in L^2_tB^0_{\infty,1}\).  In both cases the
needed compact window and logarithmic bound are derived from the velocity clock
and the model barriers, not assumed as independent hypotheses.  All criteria are
integer-Sobolev strong-solution criteria; no Leray-type weak-solution or
critical-space local theory is asserted.  The sharpness examples are static
operator obstructions for the pressure-free stress map, not dynamic blow-up
constructions.  In three dimensions the quotient has the residual split
\(T(Z)^\circ=(q_1+2aq_2)Y+q_2(Y^2)^\circ\).  Since \((Y^2)^\circ\) is generally
independent of \(Y\), the exact scalar quotient closure is intrinsically
two-dimensional.  On prescribed compact windows this residual is absorbed by
viscosity; for Oldroyd--B and FENE-P it vanishes because \(q_2\equiv0\).  No
claim is made for general anisotropic or non-spectral stresses.
\end{abstract}



\section{Introduction}

Stress-diffusion-free viscoelastic systems couple parabolic fluid smoothing to
purely transported tensor dynamics.  The velocity equation is viscous, but the
conformation tensor has no spatial diffusion and enters the momentum equation
through the divergence of an elastic stress.  A continuation criterion must
therefore identify which part of the transported tensor can feed derivatives
back into the velocity equation.  This paper shows that, in two dimensions and
for spectral isotropic stresses, the relevant object is not the full elastic
stress but its pressure quotient.

An incompressible viscoelastic equation does not use the full elastic stress as
an absolute tensor.  It uses the stress only modulo pressure: adding an
isotropic tensor \(r(C)I\) changes the scalar pressure but not the projected
velocity dynamics.  The basic object is therefore the pressure-free active
part of the stress, namely the component that survives in vorticity and can
exchange derivatives with the non-diffusive conformation tensor.  This quotient
viewpoint is especially rigid in two space dimensions.

Indeed, let \(C\in\Spp^2\) be a positive conformation tensor and let
\(T(C)\) be a smooth spectral isotropic elastic stress.  By the smooth
isotropic representation theorem, followed by the two-dimensional
Cayley--Hamilton reduction,
\[
  T(C)=q_0(\tr C,\det C)I+q_1(\tr C,\det C)C .
\]
Writing \(C=aI+Y\), with \(a=\frac12\tr C\) and \(Y=C^\circ\), gives
\[
  T(C)=(q_0+aq_1)I+q_1Y .
\]
Since \(\det C=a^2-\frac12|Y|^2\), the scalar \(q_1\) may be regarded as a
function of \((a,|Y|^2)\).  Thus, after pressure projection, every
two-dimensional spectral isotropic stress has the same active direction:
\[
  [T(C)]_{\rm active}=q_1(a,|Y|^2)Y .
\]
The pressure-free anisotropic channel is not a feature of Oldroyd--B or FENE-P;
it is the natural quotient-level form of a two-dimensional isotropic tensorial
stress.  The scalar factor carries the constitutive law, while the deviatoric
tensor \(Y\) carries the only stress direction visible to the incompressible
velocity.

This observation changes the role of model classes in the continuation
problem.  The central regularity difficulty in stress-diffusion-free
viscoelastic flow is a derivative imbalance between the viscous velocity and a
non-diffusive conformation tensor.  The elastic stress enters the momentum
equation through a divergence, so one derivative of a transported quantity is
fed back into the parabolic velocity equation.  A useful criterion should
therefore be formulated at the level of the active pressure class of the
stress, not at the level of a chosen spring law.  In two dimensions this active
pressure class is exactly a scalar spectral coefficient times \(Y\).

The first result of the paper is the corresponding active-deviatoric
projection principle.  For upper-convected stretching, the trace-free equation
has the universal principal term \(2aS(u)\).  Testing this equation with the
weight \(q_1/a\) cancels the top-order velocity--stress coupling produced by
\(\divv(q_1Y)\).  Thus the leading continuation mechanism is dictated by the
pressure quotient and by two-dimensional isotropic spectral algebra, rather
than by a special cancellation in a particular constitutive model.

This projection is exactly two-dimensional, but the three-dimensional obstruction
can be isolated rather than hidden.  In three dimensions, the Cayley--Hamilton
representation of a spectral stress generally contains a quadratic term,
\[
  T(C)=q_0I+q_1C+q_2C^2 .
\]
Writing \(C=aI+Y\), \(a=\frac13\tr C\), \(\tr Y=0\), gives the pressure-free
identity
\[
  T(C)^\circ=(q_1+2aq_2)Y+q_2(Y^2)^\circ .
\]
Thus the single active direction is replaced by a cancellative \(Y\)-channel
and a Cayley--Hamilton residual \((Y^2)^\circ\).  This is an algebraic boundary,
not just a technical loss: in \(\Sym^3_0\), \((Y^2)^\circ\) is not generally a
scalar multiple of \(Y\).  The paper treats this residual explicitly in
Section~\ref{sec:three-dimensional}: on compact conformation windows it is a
viscosity-absorbed high-order remainder, while for Oldroyd--B and FENE-P the
residual coefficient \(q_2\) is identically zero.  This is the precise sense in
which the two-dimensional quotient is exact and its three-dimensional remnant
remains controlled for the principal spectral models.  The statement is not a
classification of arbitrary three-dimensional stresses.  A non-spectral or anisotropic stress can carry pressure-free tensor
components not expressible through \(I,C,C^2\); such components require their
own coercive structure or a direct stress clock, and are outside the quotient
closure proved here.

The second result packages this quotient structure into a compact-window analytic template.
For
\[
  C=aI+Y,\qquad \eta=|Y|^2,
\]
the compact-window normal form is
\[
  \partial_tu+u\cdot\nabla u-\nu\Delta u+\nabla p
  =\alpha\divv(\tau(a,\eta)Y),
\]
\[
  D_ta+\lambda^{-1}G(a,\eta)=S(u):Y,\qquad
  D_tY+\lambda^{-1}\mu(a,\eta)Y
  =2aS(u)+[\nabla u\,Y+Y(\nabla u)^T]^\circ .
\]
Here \(\tau(a,\eta)Y\) is the pressure-free active stress allowed by
two-dimensional spectral isotropy.  The relaxation terms \(G(a,\eta)\) and
\(\mu(a,\eta)Y\) record the scalar and deviatoric relaxation channels.  The
conditions imposed on these coefficients are compact-window continuation
conditions: positive active and deviatoric relaxation coefficients, together
with monotone scalar relaxation.  They are not proposed as necessary
thermodynamic axioms for all polymer laws.  When a physical interpretation is
wanted, one should additionally impose the free-energy compatibility and
relaxation-dissipation condition in Remark~\ref{rem:thermodynamic-subclass}.
Under these hypotheses the principal cancellation uses the weight
\(\tau(a,\eta)/a\), and the high-order energy closes with coefficient
\[
  1+\norm{\nabla u}_{B^0_{\infty,1}}
  +\norm{\Log C}_{H^{1+\eps}}^2 .
\]
The compact-window theorem should therefore be read as a quotient estimate, not
as a classification of polymer models.  The concrete constitutive laws in the
paper are Oldroyd--B and FENE-P.  Remark~\ref{rem:thermodynamic-subclass}
identifies the additional free-energy compatibility needed for thermodynamic
models, and Example~\ref{ex:quadratic-channel} records a polynomial
free-energy perturbation with a genuine nonzero three-dimensional quadratic
Cayley--Hamilton coefficient.  This example calibrates the abstraction; it is
not used to promote arbitrary coefficient choices to physical models.

The concrete model criteria are obtained after the relevant compact window is
propagated.  For Oldroyd--B, the endpoint clock
\[
  \int_0^T\norm{\nabla u(t)}_{B^0_{\infty,1}}\,dt<\infty
\]
propagates the positive-cone window.  On that window a low-order estimate gives
\[
  \Log A\in L^2(0,T;H^{1+\varepsilon}).
\]
For FENE-P, the positive-cone boundary is supplemented by the
finite-extensibility boundary \(\tr C=b\); the trace-gap barrier is propagated
by the squared clock
\[
  \nabla u\in L^2_tB^0_{\infty,1}.
\]
The same pressure-free anisotropic energy then closes the high-order estimate on
this propagated FENE window.

The Besov endpoint is the zero-order velocity-gradient control that supplies the
Lagrangian Lipschitz modulus.  The criteria below are nevertheless high-order
strong-solution restart criteria.  They do not assert existence, uniqueness, or
regularization for energy-level or Leray-type weak solutions; such statements
would require a separate low-regularity framework.  Nor are the criteria
intended as critical-space local well-posedness theorems.  Their ambient class is
the standard integer Sobolev class used in the local continuation principles
below.  Thus the hypotheses are not empty: for every admissible datum in that
class with a compact initial conformation window there is a positive local
existence time, and the endpoint clock rules out breakdown of that strong
solution.

\paragraph{Conformation and dimension conventions.}
In model-specific statements we use \(A\) only for the Oldroyd--B conformation
tensor and \(C\) for the FENE-P conformation tensor.  In the abstract spectral
quotient statements, \(C\) is a dummy positive conformation tensor; after
specialization it is replaced by \(A\) for Oldroyd--B and by \(C\) for FENE-P.
The FENE-P spring factor is dimension dependent.  In dimension \(d\),
\[
  f_{b,d}(C)=\frac{b-d}{b-\tr C},\qquad
  T_{b,d}(C)=f_{b,d}(C)C-I,
  \qquad
  \calD_{b,d}=\{C\in\Spp^d:\tr C<b\},
\]
and the physical parameter range is \(b>d\).  Thus the two-dimensional sections
use \(f_{b,2}=(b-2)/(b-\tr C)\) and \(b>2\), while
Section~\ref{sec:three-dimensional} uses \(f_{b,3}=(b-3)/(b-\tr C)\) and
\(b>3\).  These are the same condition \(b>d\) in different spatial dimensions,
not two competing assumptions on one model.

The incompressible Oldroyd--B system without artificial stress diffusion is
\begin{align}
  \partial_t u+u\cdot\nabla u-\nu\Delta u+\nabla p
    &= \alpha\divv(A-\Id), \label{eq:ob-momentum}\\
  \partial_t A+u\cdot\nabla A
    &= \nabla u\,A+A(\nabla u)^T-\lambda^{-1}(A-\Id),
    \label{eq:ob-conformation}\\
  \divv u&=0.
\end{align}
Here \(u:\T^2\to\R^2\), \(A:\T^2\to\Spp^2\), and
\(\nu,\alpha,\lambda>0\).  The conformation tensor has no spatial diffusion.
The velocity is parabolic, but the force in the vorticity equation contains
\(\curl\divv A\), so derivatives of the stress enter the velocity equation at
the same level at which parabolic smoothing is being used.

The corresponding stress-diffusion-free FENE-P system is
\begin{align}
  \partial_t u+u\cdot\nabla u-\nu\Delta u+\nabla p
    &= \alpha\divv T_b(C), \label{eq:fene-momentum}\\
  \partial_t C+u\cdot\nabla C
    &= \nabla u\,C+C(\nabla u)^T
      -\lambda^{-1}\bigl(f_b(C)C-\Id\bigr),
    \label{eq:fene-conformation}\\
  \divv u&=0,
\end{align}
where, in the two-dimensional normalization used below,
\[
  f_b(C)=f_{b,2}(C)=\frac{b-2}{b-\tr C},\qquad T_b(C)=f_b(C)C-\Id .
\]
The Oldroyd--B system is recovered formally on bounded trace windows as the
Hookean limit \(b\to\infty\), where \(f_b(C)\to1\) and
\(T_b(C)\to C-\Id\).  At fixed \(b\), however, derivatives of the spring
factor contain powers of \((b-\tr C)^{-1}\), and these coefficients cannot be
read from the Oldroyd--B logarithmic variable alone.

The continuation result is stated in terms of these geometric variables, but
for Oldroyd--B the spectral window is not an additional hypothesis.  The
endpoint velocity clock
\[
  \int_0^T\norm{\nabla u(t)}_{\Besov}\,dt<\infty
\]
propagates both upper and lower eigenvalue bounds for \(A\) by a Lagrangian
comparison argument.  Once this automatic cone control is separated out,
the low-order pressure-free estimate gives
\(\Log A\in L^2(0,T;H^{1+\varepsilon})\), and the high-order estimate then
restarts the solution.  The improvement from the purely logarithmic
\(L^4_t\) closure to an \(L^2_t\) logarithmic estimate is obtained by the
pressure-free physical decomposition \(A=aI+Y\).
For FENE-P the same positive-cone geometry remains, and the finite-extensibility boundary is controlled by a scalar barrier.  The restoring force
near \(\tr C=b\) prevents trace-gap collapse under
\(\nabla u\in L^2_tB^0_{\infty,1}\).  On the resulting compact FENE window,
the low-order estimate again gives
\(\Log C\in L^2_tH^{1+\eps}_x\), and the pressure-free anisotropic energy
controls the remaining differentiability channel.

The main theorem may be summarized as follows.

\begin{mainthm}[Informal form]
In two dimensions, every spectral isotropic elastic stress has an active
deviatoric part \(q_1Y\) after pressure projection, and this active part
cancels at top order against upper-convected stretching when the trace-free
equation is tested with weight \(q_1/a\).  More generally, for compact-window
spectrally admissible active channels
\[
  \divv(\tau(a,|Y|^2)Y),\qquad
  D_tY+\lambda^{-1}\mu(a,|Y|^2)Y=\cdots ,
\]
with positive \(\tau,\mu\) and stable scalar relaxation, this gives a
compact-window endpoint estimate with coefficient
\[
  1+\norm{\nabla u}_{B^0_{\infty,1}}
  +\norm{\Log C}_{H^{1+\eps}}^2 .
\]
Consequently, a strong two-dimensional Oldroyd--B solution continues beyond
\(T\) if
\[
  \int_0^T\norm{\nabla u(t)}_{B^0_{\infty,1}}\,dt<\infty,
\]
and a strong FENE-P solution continues beyond \(T\) if
\[
  \int_0^T\norm{\nabla u(t)}_{B^0_{\infty,1}}^2\,dt<\infty .
\]
In both cases the \(L^2_tH^{1+\eps}_x\) logarithmic conformation bound is
derived on the propagated compact window rather than assumed as an external
continuation hypothesis.
The different velocity clocks come from window propagation: the Oldroyd--B
positive cone is propagated by the endpoint flow-map clock, while the FENE-P
trace gap has a finite-extensibility restoring force with boundary singularity
exponent one.  In the general trace-gap comparison, a restoring singularity
\((b-\tr C)^{-\theta}\) gives the velocity-clock exponent
\((1+\theta)/\theta\).
In three dimensions the full two-dimensional spectral reduction is no longer
available.  For a dummy tensor \(Z=aI+Y\) and spectral isotropic laws with a
smooth local Cayley--Hamilton representation on the compact window, the pressure
quotient instead splits as
\[
  T(Z)^\circ=\tau_1(Z)Y+\tau_2(Z)(Y^2)^\circ,
\]
where the second term is the Cayley--Hamilton residual.  The \(Y\)-channel is
the only channel that participates in the principal cancellation; the residual
channel is estimated separately and absorbed by viscosity on compact windows.
This compact-window statement does not cover arbitrary anisotropic or
non-spectral stresses, and the quotient algebra by itself does not propagate the
window.  For Oldroyd--B and FENE-P, however, \(\tau_2\equiv0\) and the needed
windows are propagated by the model-specific flow-map and trace-gap barriers.
Thus the principal three-dimensional model consequences are a pure velocity-clock
criterion for Oldroyd--B, a pure-clock blow-up alternative for FENE-P, and a
squared-clock window-propagated criterion for FENE-P.  Separately, if one tries
to replace this compact-window route by a fully logarithmic entropy route for
three-dimensional FENE-P, the entropy variables have a finite-dimensional
mobility threshold: they are monotone on every compact FENE window when
\(b\ge15/4\), while for \(3<b<15/4\) an additional anisotropy-controlled spectral
window is required.  This side calculation explains the parameter \(q_*\); it is
not an extra hypothesis in the velocity-clock criteria.
\end{mainthm}

\paragraph{A model-facing reading guide.}
For readers interested primarily in Oldroyd--B and FENE-P, the proof can be
read as a three-step mechanism.  First, the pressure projection removes the
isotropic stress and leaves only the deviatoric active channel.  Second, the
trace-free conformation equation contains exactly the stretching term needed to
cancel the highest-order coupling between this channel and the velocity.  Third,
the endpoint velocity clock propagates the conformation window for the model at
hand, after which the low-order logarithmic estimate and the high-order restart
estimate close the continuation argument.  Definition~\ref{def:spectral-admissible}
only records the compact-window quotient estimate behind this mechanism; it may
be skipped on a first reading of the two concrete model criteria.

\subsection{Main results}
\label{subsec:main-results}

We keep the model notation fixed: \(A\) denotes the Oldroyd--B conformation
tensor and \(C\) denotes the FENE-P conformation tensor.  Model-free spectral
lemmas use a dummy tensor, denoted by \(C\) in two dimensions and by \(Z\) in
Section~\ref{sec:three-dimensional}.  For FENE-P the Peterlin factor always uses
the ambient dimension: \(f_{b,2}=(b-2)/(b-\tr C)\) in two dimensions and
\(f_{b,3}=(b-3)/(b-\tr C)\) in three dimensions.  The first result is the
pressure-free active-deviatoric projection principle proved in
Section~\ref{sec:criterion}.  For every two-dimensional spectral isotropic
stress law, the active stress after pressure projection has the form \(q_1Y\).
This identifies the unique stress component that enters vorticity and the
top-order velocity--conformation coupling.

The second result is a compact-window quotient endpoint estimate.  For active channels
\(\tau(a,|Y|^2)Y\), coercive deviatoric relaxation
\(\mu(a,|Y|^2)Y\), and monotone scalar relaxation \(G(a,|Y|^2)\), the weighted
active-deviatoric energy closes under the coefficient
\[
  1+\norm{\nabla u}_{\Besov}
  +\norm{\Log C}_{H^{1+\eps}}^2 .
\]
This is a high-order analytic estimate on a prescribed compact window, not a
constitutive classification.  Thermodynamic polymer closures form the
subclass for which the stress and relaxation are derived from a common
free-energy density in the sense of Remark~\ref{rem:thermodynamic-subclass}.
Example~\ref{ex:quadratic-channel} gives a nontrivial member of this subclass
whose three-dimensional stress contains an actual \(C^2\) component.
The scalar-Peterlin theorem is the trace-only specialization
\[
  \tau=\psi(2a),\qquad \mu=\phi(2a),\qquad
  G=a\phi(2a)-\chi(2a),
\]
and the Oldroyd--B and FENE-P estimates are recovered from this specialization.

The third result is the Oldroyd--B continuation theorem.  A strong solution can
be continued as long as
\[
  \int_0^T\norm{\nabla u(t)}_{\Besov}\,dt<\infty .
\]
The compact positive-cone window needed by the high-order energy estimate is
not assumed; it is propagated by the velocity clock.  On this window, the
low-order pressure-free estimate gives
\(\Log A\in L^2(0,T;H^{1+\eps})\), which makes the high-order coefficient
integrable.  Equivalently, finite-time Oldroyd--B breakdown forces divergence
of the velocity Besov clock.

The fourth result is the FENE-P continuation theorem proved in
Section~\ref{sec:fene-barrier}.  The FENE-P compact window has two
boundaries:
\[
  C>0,\qquad \tr C<b.
\]
The lower spectral boundary is propagated by the velocity clock, as in
Oldroyd--B.  The finite-extensibility boundary is propagated by the scalar
barrier equation once
\(\nabla u\in L^2(0,T;B^0_{\infty,1})\).  On the resulting compact FENE window,
the pressure-free splitting \(T_b(C)=(f_ba-1)I+f_bY\) gives the same principal
anisotropic cancellation with weight \(f_b/a\).  Thus the FENE-P continuation
criterion is
\[
  \int_0^T\norm{\nabla u(t)}_{\Besov}^2\,dt<\infty .
\]
The logarithmic conformation bound needed by the high-order estimate is
derived on the propagated compact FENE window.

The fifth result is a three-dimensional compact-window extension with an
explicit Cayley--Hamilton residual split.  This result has a narrower scope than
the two-dimensional theorem.  In the model-free three-dimensional statements we
write \(Z=aI+Y\), \(a=\frac13\tr Z\), and remove scalar stress through the
pressure.  For spectral isotropic stress laws admitting a smooth local
Cayley--Hamilton representation on a compact spectral window,
\[
  T(Z)=q_0(Z)I+q_1(Z)Z+q_2(Z)Z^2,
\]
the pressure-free part is
\[
  T(Z)^\circ=(q_1+2aq_2)Y+q_2(Y^2)^\circ .
\]
The first term is the cancellative active channel.  The second term is a
quadratic residual.  Proposition~\ref{prop:3d-algebraic-boundary} shows that
this residual is generically an independent pressure-free direction, so a single
scalar quotient closure is a genuinely two-dimensional feature.  After one
integration by parts the residual pairs with the viscous velocity dissipation
and is controlled by compact-window Moser estimates.  This is not a theorem for
arbitrary anisotropic or non-spectral stresses, for which additional
pressure-free tensor directions may be present.  Nor does the algebraic split
itself propagate the compact window: the compact-window criterion is conditional
until a model-specific barrier supplies that window.
For Oldroyd--B and FENE-P the residual is absent, because their constitutive
representations are explicit with \(q_2=0\), and the required windows are
obtained from the flow-map and trace-gap estimates.  Hence a smooth
three-dimensional Oldroyd--B solution can break down only if
\[
  \int_0^{T_*}\norm{\nabla u(t)}_{\Besov}\,dt=\infty .
\]
For three-dimensional FENE-P, finite-time breakdown forces at least one of the
following: divergence of the same pure velocity clock, loss of the lower
positive-cone bound, or collapse of the trace gap \(b-\tr C\).  Under the
squared clock
\[
  \int_0^{T_*}\norm{\nabla u(t)}_{\Besov}^2\,dt<\infty,
\]
the compact FENE window is propagated and the same continuation mechanism
applies.
Finally, a separate finite-dimensional entropy-mobility calculation is recorded
for a possible three-dimensional FENE-P logarithmic formulation.  It is not used
in the velocity-clock proof.  Its role is to identify when the entropy-variable
mobility is monotone: this is automatic on compact FENE windows when
\(b\ge15/4\), while for \(3<b<15/4\) the spectral variance must remain below the
explicit threshold \(q_*(b,r)\).

The final result records a static operator obstruction associated with the
logarithmic regularity used in the proof.  For active forcing measured in
\(H^{-1+\gamma}\), the estimate based on the pressure-free stress map cannot be
closed below \(\Log C\in H^{1+\gamma}\) without additional dynamical input.
Indeed, below this threshold there are fixed-spectrum sequences for which every
zeroth-order spectral density is pointwise constant, while
\[
  \norm{\curl\divv T(C_N)}_{H^{-1+\gamma}}\to\infty .
\]
The obstruction is high-frequency rotation of the active deviatoric channel.
It explains why the logarithmic term appears in the proof, but it is not a
finite-time blow-up or dynamic optimality theorem.

The paper is organized as follows.  Section~\ref{sec:preliminaries} records the
analytic estimates used in the proof.  Section~\ref{sec:entropy} records the
positive cone and entropy structure inherited by both models.  Section~\ref{sec:log}
  gives the logarithmic equation and the good unknown.  Section~\ref{sec:criterion}
  proves the active-deviatoric projection principle, the full
  compact-window quotient estimate, its scalar-Peterlin specialization,
  and the Oldroyd--B continuation criterion, including
the propagation of the compact spectral window from the endpoint velocity
clock.  Section~\ref{sec:no-go} proves the entropy-blind obstruction and the
static \(H^{1+\gamma}\) active-forcing obstruction.
Section~\ref{sec:fene-barrier} derives the FENE-P endpoint continuation
criterion and the propagation of the
finite-extensibility trace gap.
Section~\ref{sec:three-dimensional} gives the three-dimensional
Cayley--Hamilton residual split, states its exact compact-window scope, derives
model-specific consequences for Oldroyd--B and FENE-P, and records the FENE-P
entropy-mobility boundary.
Section~\ref{sec:hookean-limit}
compares the FENE-P criterion with Oldroyd--B in the Hookean limit.

\subsection{Main mechanism and proof structure}

The pressure quotient leaves a single active direction in two dimensions.  For
\(C=aI+Y\), \(Y=C^\circ\), a smooth spectral isotropic stress has the
pressure-free form
\[
  [T(C)]_{\rm active}=q_1(a,|Y|^2)Y .
\]
Thus the same deviatoric variable that appears in the trace-free conformation
equation is the only stress component visible to the incompressible velocity.
For such compact-window quotient channels the differentiated velocity equation
contains the leading term
\[
  -\alpha\int \tau(a,\eta)\,\partial^\beta Y:S(\partial^\beta u)\,dx,
  \qquad \eta=|Y|^2,
\]
while the trace-free equation, tested with the weight
\[
  W(a,\eta)=\frac{\tau(a,\eta)}{a},
\]
contains the opposite contribution
\[
  \frac{\alpha}{2}\int W(a,\eta)\partial^\beta Y:2aS(\partial^\beta u)\,dx
  =
  \alpha\int \tau(a,\eta)\,\partial^\beta Y:S(\partial^\beta u)\,dx .
\]
The top-order velocity--stress coupling therefore cancels at the level of the
pressure-free variables.

After the endpoint transport commutators and derivatives of the weight are
estimated, the compact-window high-order energy has coefficient
\[
  1+\norm{\nabla u}_{B^0_{\infty,1}}
  +\norm{\Log C}_{H^{1+\eps}}^2 .
\]
A low-order pressure-free estimate gives the logarithmic term on the
propagated compact window.  The model-specific part of the argument is the
window propagation: Oldroyd--B uses the Lagrangian positive-cone comparison,
whereas FENE-P also uses the scalar trace-gap barrier.  This yields the
\(L^1_tB^0_{\infty,1}\) velocity clock for Oldroyd--B and the squared
\(L^2_tB^0_{\infty,1}\) clock for FENE-P.

The proof is organized accordingly.  First, the stress is reduced modulo
pressure to \(q_1Y\).  Second, the weighted trace-free energy cancels the
principal force \(\divv(q_1Y)\).  Third, the relevant compact spectral window is
propagated by the scalar comparison available in the model under consideration.
On this propagated window, the low-order estimate yields the logarithmic
regularity.  Fourth, this logarithmic bound closes the high-order estimate and
allows the strong solution to be restarted.

\subsection{Relation with Euler and stress-diffusive models}

The regularity problem for the two-dimensional Oldroyd--B system
without stress diffusion lies between two regimes that are much better
understood.  If stress diffusion is imposed, the missing derivative in the
conformation equation is restored by parabolic smoothing.  If the data are
small, perturbative mechanisms can keep the non-diffusive stress from feeding a
large cascade back into the velocity.  The regime considered here has neither source of control: no artificial
stress dissipation and no smallness.

For the two-dimensional Euler equation, vorticity is transported and the
classical obstruction is growth of the Lipschitz norm of the velocity.  For
Oldroyd--B without stress diffusion, the vorticity equation contains the extra
forcing \(\curl\divv A\).  The stress is transported and stretched, but it has
no parabolic regularization.  Thus the Euler-type vorticity criterion becomes
a coupled velocity--conformation criterion; see Corollary~\ref{cor:coupled}.
The endpoint velocity clock prevents Oldroyd--B escape from the positive cone,
while the logarithmic channel controls the remaining high-frequency
stress mechanism.

\subsection{Related work}

The Oldroyd model goes back to Oldroyd's constitutive theory
\cite{Oldroyd1950}; standard continuum and kinetic accounts include
\cite{Bird1987,BirdCurtiss1987,DoiEdwards1986,Larson1988,BerisEdwards1994,Renardy2000,OwensPhillips2002}.
Strong and weak solution theories for differential viscoelastic models were
developed in several directions, including the early work of Guillope and Saut
\cite{GuillopeSaut1990}, the global weak-solution framework of Lions and
Masmoudi \cite{LionsMasmoudi2000}, and the small-data smooth theory of Lin,
Liu, and Zhang \cite{LinLiuZhang2005}, and the boundary-value theory of Lin and
Zhang \cite{LinZhang2008}.  Lifespan and breakdown criteria for
related viscoelastic systems were studied by Chemin and Masmoudi
\cite{CheminMasmoudi2001} and by Lei, Masmoudi, and Zhou
\cite{LeiMasmoudiZhou2010}.  When additional stress diffusion is present,
two-dimensional regularity becomes substantially more tractable; see, for
example, Constantin and Kliegl \cite{ConstantinKliegl2012}, the global
regularity results of Elgindi and Rousset \cite{ElgindiRousset2015}, and the
regularized and numerical frameworks of Barrett--Boyaval and Barrett--S\"uli
\cite{BarrettBoyaval2011,BarrettSuli2007}.  The
pressure handling in these diffusive works is different from the quotient used
here.  There, Leray projection or auxiliary vorticity--stress combinations enter
parabolic or transformed energy estimates, where stress diffusion, damping, or a
model-specific transformed variable supplies the missing derivative.  Here the
isotropic stress is removed algebraically before the high-order estimate, and the
remaining active deviatoric channel is paired directly with upper-convected
stretching; no stress Laplacian is available.  In the stress-diffusion-free
system, global control is also known in small-data regimes; a recent example is
the work of Tu, Wang, and Wen \cite{TuWangWen2024}.  The results below address
the regime without stress diffusion and without smallness, using variables
adapted to the conformation geometry.

The logarithmic conformation representation was introduced in numerical
rheology by Fattal and Kupferman \cite{FattalKupferman2004,FattalKupferman2005}
and further developed in high-Weissenberg computations by Hulsen, Fattal, and
Kupferman \cite{HulsenFattalKupferman2005} to preserve positivity and improve
stability.  The
present use is analytic: the logarithmic chart separates spectral control from
high-frequency concentration.  The velocity part of the criterion follows the
spirit of the Beale--Kato--Majda continuation principle
\cite{BealeKatoMajda1984}, but the non-diffusive stress equation forces the
additional logarithmic stress diagnostic.

The finitely extensible models considered below originate from the standard
FENE closures for dilute polymeric fluids; see, for example,
\cite{BirdCurtiss1987,DoiEdwards1986}.  Mathematical work on FENE dumbbell
models has also emphasized the role of the finite-extension boundary; representative global well-posedness and weak-solution results are due to
Masmoudi \cite{Masmoudi2008,Masmoudi2013}, while finite-element and weak-solution
approaches for
regularized kinetic closures are developed in \cite{BarrettSuli2007,BarrettSuli2012}.  The FENE-P criterion below isolates the finite-extensibility boundary in an
endpoint continuation framework: the trace-gap barrier is propagated by the
squared endpoint velocity clock, and the same pressure-free active-deviatoric
cancellation closes the high-order continuation estimate.

\subsection{Relation with direct continuation criteria}

Classical non-blowup criteria for Oldroyd-type systems monitor the stress
itself.  Chemin and Masmoudi obtained criteria involving direct stress norms,
and Lei, Masmoudi, and Zhou gave a BKM-type formulation with
\(\tau\in L^1(0,T;{\rm BMO})\), together with low-integrability stress control.
The criterion here is not a norm inclusion between stress and velocity clocks.
It is a pressure-quotient certification: the isotropic stress is removed before
the high-order coupling is estimated.  This is also distinct from the pressure
manipulations used in stress-diffusive global regularity arguments, where the
projection is coupled to parabolic or transformed energy control.

This distinction is visible already for pure pressure modes.  If
\(p_N(x)=N\sin(Nx_1)\) and \(\sigma_N=p_NI\), then
\[
  \mathbb P\divv \sigma_N=0,
  \qquad
  \norm{\sigma_N}_{{\rm BMO}}\gtrsim N,
\]
where \(\mathbb P\) is the Leray projector.  A direct full-stress clock can
therefore be large because of a component that exerts no force on the
incompressible velocity.

Conversely, once the velocity clock has propagated a compact positive-cone
window for Oldroyd--B, the usual direct stress clocks are recovered.  If
\(cI\le A(t,x)\le CI\), then \(\tau=A-I\) satisfies
\[
  \norm{\tau(t)}_{{\rm BMO}}
  \le 2\norm{\tau(t)}_{L^\infty}
  \le C_{c,C},
  \qquad
  \norm{\tau(t)}_{L^1}+\norm{\tau(t)}_{L^2}
  \le C_{c,C} .
\]
Lemma~\ref{lem:oldroyd-spectral-window} shows that this window is propagated by
the endpoint velocity clock from positive initial data.  Thus, along the
solutions covered by Theorem~\ref{thm:continuation}, the direct stress
hypotheses are consequences rather than independent assumptions.

The reverse implication does not follow from direct stress information alone.
For \(A_\varepsilon=\varepsilon I\),
\[
  \norm{A_\varepsilon-I}_{{\rm BMO}}=0,
  \qquad
  \norm{A_\varepsilon-I}_{L^1}\le C,
  \qquad
  \norm{\Log A_\varepsilon}_{L^\infty}=|\log\varepsilon|\to\infty .
\]
Direct stress BMO and low-integrability bounds do not encode the lower
positive-cone boundary or the logarithmic chart.  The pressure-quotient
criterion should therefore be read as a one-way geometric mechanism: a finite
endpoint velocity clock propagates the conformation window, the low-order estimate gives the logarithmic regularity used in the high-order estimate, and
the argument then implies the
standard direct stress clocks on that window.

\section{Functional Setting and Continuation Framework}
\label{sec:functional}

We work on the periodic torus \(\T^2\) to avoid boundary effects.  The same
local estimates apply on the whole plane with the standard modifications for
decay and low frequencies.  The continuation theorems are formulated at
integer Sobolev levels \(s=m\in\mathbb N\), \(m\ge3\).  Thus the top-order
cancellations are proved by classical differentiations and the endpoint clock is
used to prevent breakdown of an already existing strong solution.  The only fractional
Sobolev input used in the argument is the low-order
\(H^{1+\eps}\) logarithmic bound, and its endpoint transport estimate is
proved below in dyadic Littlewood--Paley form.

\paragraph{Notation.}
Constants denoted by \(C\) may change from line to line.  Constants denoted by
\(C_K\) may also depend on a fixed spectral window
\(\norm{B}_{L^\infty}\le K\).  In the two-dimensional FENE-P part, \(K\) also
denotes a compact subdomain of
\[
  \calD_b=\calD_{b,2}=\{C\in\Spp^2:\tr C<b\},\qquad b>2;
\]
the constants may then depend on \(b\), the lower spectral gap of \(C\), and
the upper trace gap \(b-\tr C\).  Throughout the high-order index \(s\) is an integer \(m\ge3\).  This
keeps the continuation statement aligned with the estimates actually proved.
The fractional exponent \(1+\eps\) appearing later is handled explicitly by
the endpoint dyadic commutator Lemma~\ref{lem:endpoint-transport}; no separate
fractional high-order continuation theorem is invoked.

\begin{definition}[Strong positive-cone solution]
Let \(T>0\).  A strong positive-cone solution on \([0,T]\) is a pair
\((u,A)\) such that
\[
  u\in C([0,T];H^s),\qquad
  A\in C([0,T];H^s(\T^2;\Spp^2)),
\]
\[
  u\in L^2(0,T;H^{s+1}),\qquad \divv u=0,
\]
and \eqref{eq:ob-momentum}--\eqref{eq:ob-conformation} hold in the classical
Sobolev sense.
\end{definition}

\begin{definition}[Strong FENE-P solution]
Let \(b>2\) and \(T>0\).  A strong FENE-P solution on \([0,T]\) is a pair \((u,C)\) such
that
\[
  u\in C([0,T];H^s),\qquad
  C\in C([0,T];H^s(\T^2;\Spp^2)),
\]
\[
  u\in L^2(0,T;H^{s+1}),\qquad \divv u=0,\qquad C(t,x)\in\calD_b,
\]
and \eqref{eq:fene-momentum}--\eqref{eq:fene-conformation} hold in the
classical Sobolev sense.
\end{definition}

\begin{remark}[Strong-solution scope, low regularity, and non-vacuity]
\label{rem:strong-solution-scope}
The criteria below are continuation criteria for the strong solutions just
specified.  The assumptions are deliberately stronger than energy-level weak
formulations: the proof differentiates the equations up to order \(m\), uses
Moser calculus for smooth spectral functions of the conformation tensor, and
uses pointwise positive-cone and finite-extensibility barriers.  These operations
are not available for a general Leray-type weak solution without additional
renormalization, weak--strong uniqueness, or compactness machinery.  Therefore
no low-regularity existence or weak-solution regularity theorem is claimed here.

This restriction does not make the criteria vacuous.  The local continuation
principles below give the ambient local theory used throughout the paper: if the
initial data lie in \(H^m\), \(m\ge3\) in two dimensions, with the conformation
tensor in a compact positive-cone window, and in the FENE-P case also in a
compact finite-extensibility window, then a strong solution exists on a positive
time interval.  The continuation criteria are therefore statements about the
maximal lifespan in this nonempty strong class.  Failure of a critical \(H^1\)-
or energy-level local theory would only mean that the present theorem is not
formulated at that lower regularity; it would not affect the integer-Sobolev
restart criterion.

If a weak or energy solution is known by some other argument to coincide with a
strong positive-cone solution on an interval, then the present endpoint clock
prevents blow-up within that strong class; it does not by itself upgrade a weak
solution to a strong one.
\end{remark}

\begin{proposition}[Local continuation principle]
\label{prop:local-continuation}
Let \(s\in\mathbb N\), \(s\ge3\), \(u_0\in H^s\), \(\divv u_0=0\), and
\(A_0\in H^s(\T^2;\Spp^2)\) with spectral range contained in
\([m_0,M_0]\), \(0<m_0<M_0<\infty\).  Then there exists a unique strong
positive-cone solution on a time interval \([0,T_{\rm loc}]\).  Moreover, if a
solution on \([0,T_*)\) satisfies
\[
  \sup_{t<T_*}\left(
  \norm{u(t)}_{H^s}+\norm{A(t)}_{H^s}
  +\norm{A(t)}_{L^\infty}
  +\norm{A(t)^{-1}}_{L^\infty}
  \right)<\infty,
  \label{eq:local-continuation}
\]
then it extends as a strong positive-cone solution beyond \(T_*\).
\end{proposition}

\begin{proof}
Regularize the system, solve by Picard iteration, and obtain estimates in
\(H^s\) for \(A\) and in \(H^s\cap L^2_tH^{s+1}\) for \(u\).  Positivity is
propagated along characteristics by Lemma~\ref{lem:cone}.  The lower spectral
bound prevents degeneration of the cone, while the upper spectral and Sobolev
bounds control all nonlinear coefficients.

More explicitly, the regularized estimates give a local existence time
\[
  T_{\rm loc}\ge
  \tau\!\left(
  \norm{u_0}_{H^s}+\norm{A_0}_{H^s},
  \norm{A_0}_{L^\infty},
  \norm{A_0^{-1}}_{L^\infty}
  \right)>0,
\]
where \(\tau\) is nonincreasing in its arguments.  In particular, the
integer-Sobolev positive-cone class contains a full local solution class for
every datum satisfying the displayed bounds.  The same estimate applies
when the construction is restarted from any time \(t_0<T_*\).  Under
\eqref{eq:local-continuation}, these arguments remain bounded uniformly for
\(t_0<T_*\).  Hence there is a uniform \(\tau_*>0\) such that the solution
restarts on \([t_0,t_0+\tau_*]\) for every \(t_0<T_*\) sufficiently close to
\(T_*\).  Choosing \(t_0>T_*-\tau_*/2\) gives an extension beyond \(T_*\).
\end{proof}

\begin{remark}[Why logarithmic coordinates appear in the estimates]
The local principle is naturally stated in \(A\), but the compact-window
estimates are most transparent in \(B=\Log A\).  Uniform upper and lower
spectral bounds for \(A\) are equivalent to an \(L^\infty\) bound for \(B\).
On such spectral windows, the maps \(B\mapsto e^B\) and \(A\mapsto\Log A\) are
smooth on Sobolev spaces.  Therefore controlling \(B\) in \(H^s\) is equivalent
to controlling \(A\) in \(H^s\), with constants depending only on the spectral
window.
\end{remark}

\begin{proposition}[Local continuation in the FENE-P domain]
\label{prop:fene-local-continuation}
Let \(b>2\), \(s\in\mathbb N\), \(s\ge3\), \(u_0\in H^s\), \(\divv u_0=0\), and
\(C_0\in H^s(\T^2;\calD_b)\).  Assume that the range of \(C_0\) is contained
in a compact set \(K_0\Subset\calD_b\).  Then there exists a unique strong
FENE-P solution on a time interval \([0,T_{\rm loc}]\).  Moreover, if a
solution on \([0,T_*)\) satisfies
\[
  \sup_{t<T_*}\left(
  \norm{u(t)}_{H^s}+\norm{C(t)}_{H^s}
  \right)<\infty
\]
and the range of \(C(t,\cdot)\) remains in a fixed compact set
\(K\Subset\calD_b\) for all \(t<T_*\), then the solution extends beyond
\(T_*\) as a strong FENE-P solution.
\end{proposition}

\begin{proof}
On \(K\Subset\calD_b\), the maps
\[
  C\mapsto f_b(C),\qquad C\mapsto T_b(C),\qquad
  C\mapsto f_b(C)C-\Id
\]
are smooth with bounded derivatives of every finite order.  The FENE-P system
therefore has the same quasilinear parabolic-transport structure as the
Oldroyd--B system, with composition constants depending only on \(K\), \(b\),
and the \(H^s\) size of \(C\).  Regularization and Picard iteration give a
local existence time bounded below by a nonincreasing function of
\[
  \norm{u_0}_{H^s}+\norm{C_0}_{H^s}
\]
and the compact-set constants of \(K_0\).  Thus the FENE-P strong class is
a genuine local solution class for compact finite-extensibility data, not an
additional a priori assumption.  Restarting the construction at
times \(t_0<T_*\) gives the same lower bound with \(K\) in place of \(K_0\).
The stated uniform bounds therefore give a uniform restart time and extend the
solution beyond \(T_*\).
\end{proof}

\section{Analytic Preliminaries}
\label{sec:preliminaries}

We collect the analytic estimates used below.  We use the standard
Littlewood--Paley and paraproduct framework of
\cite{BahouriCheminDanchin2011} and the commutator estimates of
\cite{KatoPonce1988}.  The endpoint transport estimate is written in dyadic form
because it is the only place where the critical clock \(B^0_{\infty,1}\)
interacts with a fractional Sobolev index.
The remaining high-order coefficient estimates are used at integer levels.

\begin{lemma}[Standard Moser and tame calculus]
\label{lem:moser}
\label{lem:tame-high-low}
Let \(r>1\) in two dimensions.  Then \(H^r(\T^2)\) is an algebra and
\[
  \norm{fg}_{H^r}
  \le C_r\left(\norm{f}_{L^\infty}\norm{g}_{H^r}
  +\norm{g}_{L^\infty}\norm{f}_{H^r}\right).
  \label{eq:moser-product}
\]
If \(F\in C^\infty(\R^N)\), \(F(0)=0\), and
\(\norm{v}_{L^\infty}\le K\), then
\[
  \norm{F(v)}_{H^r}\le C_{F,K,r}\norm{v}_{H^r}.
  \label{eq:moser-composition}
\]
The same estimates hold componentwise for matrix-valued functions.
Moreover, if \(m\ge2\), \(0<\eps<1\), and
\(f_1,\ldots,f_N\) be smooth scalar, vector, or matrix fields on \(\T^2\).  If
at most one factor is measured at order \(m\), then
\[
  \norm{\prod_{j=1}^N f_j}_{H^m}
  \le
  C\prod_{j=1}^N\left(1+\norm{f_j}_{H^{1+\eps}}\right)
  \sum_{j=1}^N\norm{f_j}_{H^m}.
  \label{eq:tame-high-low}
\]
The constant depends on \(m,\eps,N\) and on fixed \(L^\infty\) bounds for
smooth coefficient functions of the \(f_j\)'s.
\end{lemma}

\begin{proof}
These are standard consequences of the Bony paraproduct decomposition
\(fg=T_f g+T_g f+R(f,g)\).  Since \(r>1=d/2\), \(H^r(\T^2)\) is an algebra.
The composition bound follows by writing
\[
  F(v)=\int_0^1 DF(\theta v)v\,d\theta .
\]
For \eqref{eq:tame-high-low}, put the highest derivative on one factor and
place all remaining factors in \(H^{1+\eps}\hookrightarrow L^\infty\).
\end{proof}

\begin{lemma}[Endpoint Besov transport commutator]
\label{lem:endpoint-transport}
Let \((\Delta_j)_{j\ge -1}\) be a nonhomogeneous Littlewood--Paley partition
on \(\T^2\).  Let \(\sigma>0\), \(\divv u=0\), and let \(f\) be scalar,
vector, or matrix valued.  Then
\[
  \left|
  \sum_{j\ge -1}2^{2\sigma j}
  \int_{\T^2}\Delta_j f:
  [\Delta_j,u\cdot\nabla]f\,dx
  \right|
  \le
  C_\sigma\norm{\nabla u}_{B^0_{\infty,1}}\norm{f}_{H^\sigma}^2.
  \label{eq:endpoint-transport-dyadic}
\]
Consequently, for every integer \(m\ge1\),
\[
  \left|
  \sum_{|\gamma|\le m}
  \int_{\T^2}\partial^\gamma f:
  [\partial^\gamma,u\cdot\nabla]f\,dx
  \right|
  \le
  C_m\norm{\nabla u}_{B^0_{\infty,1}}\norm{f}_{H^m}^2.
  \label{eq:endpoint-transport}
\]
\end{lemma}

\begin{proof}
The proof is included to make the endpoint use of \(B^0_{\infty,1}\) explicit.
Set \(b_k=\norm{\Delta_k\nabla u}_{L^\infty}\), so that
\(\sum_{k\ge-1}b_k\simeq\norm{\nabla u}_{B^0_{\infty,1}}\).  Bony's
splitting of \(u\cdot\nabla f\) gives
\[
  [\Delta_j,u\cdot\nabla]f
  =R_j^1+R_j^2+R_j^3,
\]
where, up to harmless finite changes of the index ranges,
\[
  R_j^1=\sum_{|k-j|\le4}
  [\Delta_j,S_{k-1}u\cdot\nabla]\Delta_k f,
\]
\(R_j^2\) contains the high-frequency velocity acting on the low-frequency
part of \(f\), and \(R_j^3\) contains the high-high remainder.  The kernel
representation of \(\Delta_j\) gives
\[
  \norm{R_j^1}_{L^2}
  \le C\sum_{|k-j|\le4}
  \norm{\nabla S_{k-1}u}_{L^\infty}\norm{\Delta_k f}_{L^2}
  \le C\norm{\nabla u}_{B^0_{\infty,1}}
  \sum_{|k-j|\le4}\norm{\Delta_k f}_{L^2}.
\]
For the two remaining terms, the derivative on the high-frequency velocity is
recorded in the summable sequence \(b_k\), and the dyadic interactions are
localized by \(\ell^1\) kernels.  More precisely,
\[
  2^{\sigma j}\norm{R_j^2+R_j^3}_{L^2}
  \le C_\sigma
  \sum_{k\ge-1}2^{-c_\sigma|j-k|}
  b_k\,2^{\sigma k}\norm{\Delta_k f}_{L^2},
  \label{eq:lp-commutator-kernel}
\]
for some \(c_\sigma>0\).  The estimate follows by Young's inequality for
sequences, using \(b\in\ell^1\).  Equivalently, there is a sequence
\(c_j\in\ell^2\), \(\norm{c}_{\ell^2}\le1\), such that
\[
  2^{\sigma j}\norm{[\Delta_j,u\cdot\nabla]f}_{L^2}
  \le C_\sigma c_j
  \norm{\nabla u}_{B^0_{\infty,1}}\norm{f}_{H^\sigma}.
\]
Multiplying by \(2^{\sigma j}\norm{\Delta_j f}_{L^2}\) and summing in
\(j\) gives \eqref{eq:endpoint-transport-dyadic}.  The divergence-free
condition has removed the principal transport contribution
\(\int u\cdot\nabla |\Delta_j f|^2/2=0\).  The integer derivative form
\eqref{eq:endpoint-transport} follows from the equivalence between the
nonhomogeneous dyadic \(H^m\) norm and the classical derivative norm.
\end{proof}

\begin{remark}[Scope of fractional estimates]
\label{rem:fractional-scope}
Lemma~\ref{lem:endpoint-transport} is the fractional Littlewood--Paley input
used in the paper.  The high-order continuation estimates themselves are
proved for integer \(s=m\ge3\), where all coefficient commutators are
expanded by ordinary derivatives.  Thus the proof does not rely on an unstated
fractional high-order commutator calculus for the full nonlinear system, nor on
a low-regularity passage below the strong-solution class of
Remark~\ref{rem:strong-solution-scope}.
\end{remark}

\section{Entropy and Positive-Cone Geometry}
\label{sec:entropy}

\begin{definition}[Free energy]
For a smooth positive-cone solution define
\[
  \calE(t)
  =
  \frac12\int_{\T^2}|u|^2\dd
  +
  \frac{\alpha}{2}\int_{\T^2}
  \left(\tr A-\log\det A-2\right)\dd .
\]
\end{definition}

\begin{proposition}[Entropy identity]
\label{prop:entropy}
Every smooth solution of \eqref{eq:ob-momentum}--\eqref{eq:ob-conformation}
satisfies
\[
\begin{aligned}
  \frac{d}{dt}\calE(t)
  &+\nu\int_{\T^2}|\nabla u|^2\dd
  +\frac{\alpha}{2\lambda}
  \int_{\T^2}\tr(A+A^{-1}-2\Id)\dd=0 .
\end{aligned}
\label{eq:entropy}
\]
\end{proposition}

\begin{proof}
Testing the momentum equation by \(u\) gives
\[
  \frac12\frac{d}{dt}\int |u|^2\dd
  +\nu\int|\nabla u|^2\dd
  =
  -\alpha\int (A-\Id):\nabla u\dd .
\]
Testing the conformation equation by \(\Id-A^{-1}\) gives
\[
  \frac{d}{dt}\int(\tr A-\log\det A-2)\dd
  =
  2\int(A-\Id):\nabla u\dd
  -\lambda^{-1}\int\tr(A+A^{-1}-2\Id)\dd .
\]
The stretching terms cancel after multiplication by \(\alpha/2\).
\end{proof}

\begin{lemma}[Cone propagation]
\label{lem:cone}
If \(A_0(x)\in\Spp^2\) and
\(\int_0^T\norm{\nabla u(t)}_{L^\infty}\,dt<\infty\), then
\(A(t,x)\in\Spp^2\) for all \(t\le T\).  Moreover, along characteristics,
the extreme eigenvalues obey
\[
  \lambda_{\max}A(t)
  \le C\left(\lambda_{\max}A_0+1\right)
  \exp\left(2\int_0^t\norm{\nabla u(s)}_{L^\infty}\,ds\right),
\]
and an analogous lower bound holds for \(\lambda_{\min}A(t)\).
\end{lemma}

\begin{proof}
Along a trajectory \(X(t)\), the matrix satisfies
\[
  \frac{d}{dt}A=\nabla u\,A+A(\nabla u)^T-\lambda^{-1}(A-\Id).
\]
Testing against unit vectors and applying Gronwall gives the upper spectral
bound.  Applying the same argument to \(A^{-1}\), whose equation is obtained by
differentiating \(A^{-1}A=\Id\), gives the lower bound.
\end{proof}

\section{Logarithmic Coordinates and Good Unknowns}
\label{sec:log}

Let \(B=\Log A\).  The derivative of the matrix exponential is
\[
  d\exp_B(C)=\int_0^1 e^{(1-\theta)B}Ce^{\theta B}\,d\theta .
\]
If the spectrum of \(B\) is contained in \([-K,K]\), then
\[
  c_K|C|\le |d\exp_B(C)|\le C_K|C|
  \qquad (C=C^T),
  \label{eq:dexp-elliptic}
\]
and the constants depend only on \(K\).  Thus the logarithmic coordinate is
not merely a change of notation: it gives an elliptic chart on every compact
spectral subset of the positive cone.

\begin{lemma}[Sobolev calculus for the exponential chart]
\label{lem:exp-calculus}
Let \(s>1\) in two dimensions and suppose
\(\norm{B}_{L^\infty}\le K\).  Then
\[
  \norm{e^B-\Id}_{H^s}
  \le C_K\left(1+\norm{B}_{H^s}\right),
  \label{eq:exp-Hs}
\]
and, for two symmetric fields \(B_1,B_2\) with the same \(L^\infty\) bound,
\[
  \norm{e^{B_1}-e^{B_2}}_{H^{s-1}}
  \le
  C_K\left(1+\norm{B_1}_{H^s}+\norm{B_2}_{H^s}\right)
  \norm{B_1-B_2}_{H^{s-1}} .
  \label{eq:exp-lip}
\]
The same estimates hold for \(B\mapsto(d\exp_B)^{-1}\) on symmetric matrices.
\end{lemma}

\begin{proof}
Use the integral representation
\[
  e^B-\Id=\int_0^1 e^{\theta B}B\,d\theta
\]
and Lemma~\ref{lem:moser}.  The Lipschitz bound follows from
\[
  e^{B_1}-e^{B_2}
  =
  \int_0^1
  d\exp_{B_2+\theta(B_1-B_2)}(B_1-B_2)\,d\theta .
\]
The inverse derivative is controlled by \eqref{eq:dexp-elliptic} and the
smooth dependence of the inverse operator on \(B\) on compact spectral sets.
\end{proof}

In logarithmic variables, the conformation equation can be written as
\[
  d\exp_B\left(\partial_tB+u\cdot\nabla B\right)
  =
  \nabla u\,e^B+e^B(\nabla u)^T-\lambda^{-1}(e^B-\Id).
  \label{eq:log-equation}
\]
When the spectrum of \(B\) is bounded, \(d\exp_B\) and its inverse are bounded
on symmetric matrices.  Thus the equation is transport-dominated, but the
source is nonlinear in \(B\) and \(\nabla u\).

The naive differentiated equation contains two bad-looking terms.  The first
is the standard transport commutator
\([\partial^\gamma,u\cdot\nabla]B\), which is controlled by a critical Besov
norm of \(\nabla u\).  The second is more specific to the positive cone: when
one differentiates \(e^B\), the top derivative does not simply equal
\(e^B\partial^\gamma B\).  It contains non-commuting matrix factors and
lower-order products.  The good unknown below subtracts the paralinear part of
this matrix-exponential commutator.  Geometrically, it applies the inverse
tangent map of the exponential chart and therefore pulls the differentiated
conformation tensor \(\partial^\gamma A\) back to the logarithmic tangent
space at \(B\).  This is the point at which the positive-cone geometry enters
the continuation argument.

\begin{definition}[Logarithmic good unknown]
For a multi-index \(\gamma\), define
\[
  \calG_\gamma
  =
  (d\exp_B)^{-1}\left(\partial^\gamma(e^B)\right).
  \label{eq:good-unknown}
\]
By Lemma~\ref{lem:commutator-structure},
\[
  \calG_\gamma
  =
  \partial^\gamma B
  +(d\exp_B)^{-1}\calC_\gamma(B).
\]
Thus \(\calG_\gamma\) is an Alinhac-type correction of
\(\partial^\gamma B\): it is defined intrinsically from the differentiated
conformation tensor and differs from \(\partial^\gamma B\) only by lower-order
matrix-exponential commutators.
\end{definition}

\begin{lemma}[Commutator structure]
\label{lem:commutator-structure}
Let \(m\ge2\), \(0<\eps<1\), and assume \(\norm{B}_{L^\infty}\le K\).  Then
for \(|\gamma|=m\),
\[
  \partial^\gamma(e^B)
  =
  d\exp_B(\partial^\gamma B)+\calC_\gamma(B),
  \label{eq:exp-commutator}
\]
where
\[
  \norm{\calC_\gamma(B)}_{L^2}
  \le
  C_K\left(1+\norm{B}_{H^{1+\eps}}^2\right)
  \left(1+\norm{B}_{H^{m-1}}\right).
  \label{eq:commutator-bound}
\]
\end{lemma}

\begin{proof}
Expanding \(\partial^\gamma(e^B)\) by repeated differentiation of the integral
formula for \(d\exp_B\), the unique term containing \(m\) derivatives of \(B\)
is \(d\exp_B(\partial^\gamma B)\).  Every other term is a finite product with
at least two positive-order derivatives of \(B\) and total order at most \(m\).
In two dimensions the \(H^{1+\eps}\) factor controls the low-order products,
while the remaining derivative is placed in \(H^{m-1}\).  The spectral bound
absorbs the exponential coefficients.
\end{proof}

\begin{proposition}[Good-unknown estimate]
\label{prop:good-unknown}
Let \(s\in\mathbb N\), \(s\ge3\), \(0<\varepsilon<1\), and let \((u,B)\) be smooth.  If \(B\)
has bounded spectrum on \([0,T]\), then for every \(2\le m\le s\),
\[
  \frac{d}{dt}\norm{B}_{H^m}
  \le
  C\left(1+\norm{\nabla u}_{\Besov}
  +\norm{B}_{H^{1+\varepsilon}}^2\right)
  \left(1+\norm{B}_{H^m}+\norm{u}_{H^{m+1}}\right).
\label{eq:good-est}
\]
\end{proposition}

\begin{proof}
Write the logarithmic equation as
\[
  \partial_tB+u\cdot\nabla B
  =
  \mathfrak F(B,\nabla u),
  \qquad
  \mathfrak F(B,M)
  =
  (d\exp_B)^{-1}
  \left(Me^B+e^BM^T-\lambda^{-1}(e^B-\Id)\right).
\]
On a fixed spectral window, Lemma~\ref{lem:exp-calculus} and
Lemma~\ref{lem:tame-high-low} give the tame bound
\[
  \norm{\mathfrak F(B,\nabla u)}_{H^m}
  \le
  C_K\left(1+\norm{B}_{H^{1+\eps}}^2\right)
  \left(1+\norm{B}_{H^m}+\norm{u}_{H^{m+1}}\right).
  \label{eq:F-tame}
\]
Indeed the top derivative falls either on \(\nabla u\), giving
\(\norm{u}_{H^{m+1}}\), or on one copy of \(B\) in the exponential
coefficients, giving a term linear in \(\norm{B}_{H^m}\); all remaining factors
are placed in \(H^{1+\eps}\hookrightarrow L^\infty\).  Differentiating the
transport equation and pairing with \(\partial^\gamma B\) yields
\[
  \frac12\frac{d}{dt}\norm{\partial^\gamma B}_{L^2}^2
  =
  -\int_{\T^2}\partial^\gamma B\,
  [\partial^\gamma,u\cdot\nabla]B\,dx
  +
  \int_{\T^2}\partial^\gamma B\,\partial^\gamma\mathfrak F\,dx .
\]
The endpoint transport commutator is bounded by Lemma~\ref{lem:endpoint-transport}:
\[
  \left|
  \int_{\T^2}\partial^\gamma B\,
  [\partial^\gamma,u\cdot\nabla]B\,dx
  \right|
  \le
  C\norm{\nabla u}_{\Besov}\norm{B}_{H^m}^2 .
\]
After summing over \(|\gamma|\le m\), using \eqref{eq:F-tame}, and taking the
square root of the resulting differential inequality, we obtain
\eqref{eq:good-est}.  The good unknown \(\calG_\gamma\) gives an equivalent
way to phrase the same estimate in the differentiated conformation variable
\(\partial^\gamma A\): by Lemma~\ref{lem:commutator-structure},
\(\calG_\gamma-\partial^\gamma B\) is lower order and is absorbed by the
\(\norm{B}_{H^{1+\eps}}^2(1+\norm{B}_{H^m})\) term.
\end{proof}

\section{Active-Deviatoric Projection and Endpoint Continuation}
\label{sec:criterion}

Let \(\omega=\curl u\).  The endpoint velocity modulus is controlled by the
periodic Biot--Savart law,
\[
  \norm{\nabla u}_{B^0_{\infty,1}}
  \le C\left(\norm{\omega}_{B^0_{\infty,1}}+\norm{u}_{L^2}\right).
\]

The logical dependence of the argument is as follows.
\begin{center}
\small
\begin{tabular}{>{\raggedright\arraybackslash}p{0.21\textwidth}
                >{\raggedright\arraybackslash}p{0.29\textwidth}
                >{\raggedright\arraybackslash}p{0.40\textwidth}}
\hline
Layer & Main statement & Role \\
\hline
Two-dimensional spectral stress &
Proposition~\ref{prop:active-deviatoric-projection} &
Identifies the pressure-free active stress \(q_1Y\) and the principal
deviatoric cancellation. \\
Compact-window quotient template &
Definition~\ref{def:spectral-admissible} and
Proposition~\ref{prop:spectral-admissible-endpoint} &
Gives the compact-window endpoint energy estimate for active channels
\(\tau(a,|Y|^2)Y\). \\
Logarithmic regularity estimate &
Proposition~\ref{prop:compact-window-log-generation} &
Derives \(\Log C\in L^2_tH^{1+\eps}_x\) on compact windows from the endpoint
velocity clock. \\
Scalar-Peterlin specialization &
Definition~\ref{def:admissible-scalar-peterlin} and
Proposition~\ref{prop:scalar-peterlin-endpoint} &
Recovers \(T_\psi(C)=\psi(\tr C)C-I\), Oldroyd--B, and FENE-P inside the full
framework. \\
Window propagation &
Lemma~\ref{lem:oldroyd-spectral-window} and
Lemma~\ref{lem:finite-extensibility-clock-exponent} &
Converts model-specific scalar geometry into the compact windows required by
the endpoint estimate. \\
Concrete models &
Theorem~\ref{thm:continuation} and
Theorem~\ref{thm:unified-continuation} &
Specializes the general estimate to Oldroyd--B and FENE-P continuation
criteria. \\
Obstructions &
Theorem~\ref{thm:no-go}, Proposition~\ref{prop:spectral-energy-active-blind},
and Theorem~\ref{thm:active-threshold} &
Shows that zeroth-order spectral or entropy information cannot replace the
logarithmic differentiability used by the pressure-free stress estimate. \\
\hline
\end{tabular}
\end{center}

We begin with the structural projection which is independent of the particular
Oldroyd--B or FENE-P closure.

\begin{proposition}[Pressure-free active-deviatoric projection]
\label{prop:active-deviatoric-projection}
Let \(T:\Spp^2\to\Sym^2\) be a smooth spectral isotropic stress law in two
space dimensions:
\[
  T(QCQ^T)=QT(C)Q^T,\qquad Q\in O(2).
\]
For every compact \(K\Subset\Spp^2\), there are smooth scalar functions
\(q_0,q_1\), depending on \(s=\tr C\) and \(d=\det C\), such that
\[
  T(C)=q_0(\tr C,\det C)I+q_1(\tr C,\det C)C .
\]
Writing \(C=aI+Y\), \(a=\frac12\tr C\), \(Y=C^\circ\), one has
\[
  T(C)=(q_0+aq_1)I+q_1Y,\qquad T(C)^\circ=q_1Y .
\]
Consequently, in the incompressible momentum equation, the isotropic part is a
pressure mode and only \(\divv(q_1Y)\) enters the vorticity and the top-order
velocity--conformation coupling.
If the conformation equation has upper-convected stretching
\[
  D_tC=\nabla u\,C+C(\nabla u)^T+R(C),
\]
then the principal active coupling cancels when the trace-free equation is
tested with weight \(q_1/a\).
\end{proposition}

\begin{proof}
Spectral isotropy implies that \(T(C)\) and \(C\) are coaxial.  If \(C\) has
distinct eigenvalues \(\lambda_1,\lambda_2\) and \(T(C)\) has corresponding
eigenvalues \(t_1,t_2\), define
\[
  q_1=\frac{t_1-t_2}{\lambda_1-\lambda_2},\qquad
  q_0=\frac{\lambda_1t_2-\lambda_2t_1}{\lambda_1-\lambda_2}.
\]
The smooth spectral representation theorem, equivalently the smooth
divided-difference lemma for symmetric spectral functions, extends \(q_0,q_1\)
smoothly across the repeated-eigenvalue set.  Since symmetric smooth functions
of \(\lambda_1,\lambda_2\) are smooth functions of \(\tr C\) and \(\det C\) on
compact positive-cone windows, the displayed representation follows.

The pressure projection is then algebraic:
\[
  \divv T(C)=\nabla(q_0+aq_1)+\divv(q_1Y).
\]
The first term is absorbed into the pressure and has zero curl.  Finally, the
trace-free part of the upper-convected stretching is
\[
  \{\nabla u\,C+C(\nabla u)^T\}^\circ
  =2aS(u)+[\nabla u\,Y+Y(\nabla u)^T]^\circ .
\]
At derivative order \(\beta\), the velocity equation contributes
\[
  -\alpha\int q_1\,\partial^\beta Y:S(\partial^\beta u)\,dx,
\]
while testing the trace-free equation by
\((\alpha/2)(q_1/a)\partial^\beta Y\) gives
\[
  \frac{\alpha}{2}\int \frac{q_1}{a}\partial^\beta Y:
  2aS(\partial^\beta u)\,dx
  =
  \alpha\int q_1\,\partial^\beta Y:S(\partial^\beta u)\,dx .
\]
The principal terms cancel exactly.
\end{proof}

We next record the compact-window quotient estimate abstracted from the
active-deviatoric projection above.  This is an analytic template for the
high-order proof, not a proposal for a new class of physical constitutive laws.

\begin{definition}[Compact-window quotient closure]
\label{def:spectral-admissible}
Let \(K\Subset\Spp^2\) be a compact conformation window and write
\[
  C=aI+Y,\qquad a=\frac12\tr C,\qquad Y=C^\circ,\qquad \eta=|Y|^2 .
\]
A compact-window closure is spectrally admissible on \(K\) if, in the variables
\((a,Y)\), it has the form
\begin{align}
  \partial_tu+u\cdot\nabla u-\nu\Delta u+\nabla p
  &=\alpha\divv(\tau(a,\eta)Y),
  \label{eq:spectral-active-u}\\
  D_ta+\lambda^{-1}G(a,\eta)&=S(u):Y,
  \label{eq:spectral-active-a}\\
  D_tY+\lambda^{-1}\mu(a,\eta)Y
  &=2aS(u)+[\nabla u\,Y+Y(\nabla u)^T]^\circ .
  \label{eq:spectral-active-Y}
\end{align}
The coefficients are smooth on the range of \((a,\eta)\) determined by \(K\).
Moreover, there are constants \(a_0,\tau_0,\mu_0,\kappa>0\) and \(a_*\) such that
\[
  a\ge a_0,\qquad \tau(a,\eta)\ge\tau_0,\qquad
  \mu(a,\eta)\ge\mu_0,\qquad G(a_*,0)=0,
  \qquad \partial_aG(a,\eta)\ge\kappa
\]
on this range.  These are compact-window continuation assumptions.  They do
not by themselves assert a thermodynamic free-energy law for arbitrary
\(\tau,\mu,G\).
\end{definition}

\begin{remark}[Free-energy-compatible subclass]
\label{rem:thermodynamic-subclass}
Definition~\ref{def:spectral-admissible} is a continuation class.  A sufficient
condition for it to come from a thermodynamic polymer closure is the existence
of a smooth free-energy density
\(\mathfrak F=\mathfrak F(a,\eta)\) on the compact window such that
\begin{equation}
  \tau(a,\eta)=\frac12\partial_a\mathfrak F(a,\eta)
  +2a\partial_\eta\mathfrak F(a,\eta),
  \label{eq:thermo-stress-compatibility}
\end{equation}
and the relaxation dissipates the same density,
\begin{equation}
  \partial_a\mathfrak F(a,\eta)G(a,\eta)
  +2\partial_\eta\mathfrak F(a,\eta)\mu(a,\eta)\eta\ge0
  \label{eq:thermo-relax-compatibility}
\end{equation}
on the window.  Indeed, for a spectral density \(\mathfrak F(a,\eta)\),
\(\partial_C\mathfrak F=(\partial_a\mathfrak F/2)I
+2\partial_\eta\mathfrak F\,Y\), and in two dimensions
\(Y^2=(\eta/2)I\).  Hence the pressure-free part of
\(C\partial_C\mathfrak F\) is exactly
\((\partial_a\mathfrak F/2+2a\partial_\eta\mathfrak F)Y\), which gives
\eqref{eq:thermo-stress-compatibility}.  The relaxation part of
\eqref{eq:spectral-active-a}--\eqref{eq:spectral-active-Y} gives
\(D_ta=-\lambda^{-1}G\) and
\(D_t\eta=-2\lambda^{-1}\mu\eta\); therefore
\eqref{eq:thermo-relax-compatibility} is precisely the local entropy
production inequality for \(\mathfrak F\).

The endpoint theorem below does not require
\eqref{eq:thermo-stress-compatibility}--\eqref{eq:thermo-relax-compatibility};
it requires only the continuation assumptions in
Definition~\ref{def:spectral-admissible}.  Thus the abstract theorem is not
presented as a classification of all thermodynamic isotropic polymer models.
It is a pressure-quotient continuation mechanism, and thermodynamic closures
are recovered by adding the compatibility conditions above.  The Oldroyd--B and
FENE-P specializations satisfy these conditions with the standard positive-cone
free energies.

\end{remark}

\begin{example}[A free-energy-compatible channel with a genuine quadratic component]
\label{ex:quadratic-channel}
Let \(0<\theta<1\) and define, for \(C\in\Spp^d\),
\[
  \mathfrak F_\theta(C)
  =\tr(C-\Log C-I)+\frac{\theta}{2}\tr(C-I)^2 .
\]
Take the elastic stress and relaxation to be
\[
  T_\theta(C)=C\,\partial_C\mathfrak F_\theta(C)
  =C-I+\theta(C^2-C),
  \qquad R(C)=-(C-I).
\]
The relaxation dissipates the same density, since
\[
  \partial_C\mathfrak F_\theta(C):(C-I)
  =\tr(C+C^{-1}-2I)+\theta |C-I|^2\ge0.
\]
In two dimensions, writing \(C=aI+Y\), \(\eta=|Y|^2\), this example gives
\[
  G(a,\eta)=a-1,
  \qquad \mu(a,\eta)=1,
  \qquad \tau_\theta(a,\eta)=1+\theta(2a-1).
\]
Hence \(\partial_aG=1\), \(\mu=1\), and
\(\tau_\theta\ge 1-\theta>0\) on the positive cone.  It is therefore a member of
Definition~\ref{def:spectral-admissible} and satisfies the compatibility
conditions in Remark~\ref{rem:thermodynamic-subclass}.  In three dimensions its
local Cayley--Hamilton representation is explicit,
\[
  T_\theta(C)=-I+(1-\theta)C+\theta C^2,
\]
so \(q_2\equiv\theta\ne0\), and the pressure-free split contains the genuine
residual channel
\[
  T_\theta(C)^\circ=(1-\theta+2a\theta)Y+\theta (Y^2)^\circ .
\]
This example is included to calibrate the abstract compact-window mechanism; it
is not used as an additional physical model theorem.
\end{example}

\begin{proposition}[Compact-window quotient endpoint estimate]
\label{prop:spectral-admissible-endpoint}
Let \(m\ge3\), \(0<\eps<1\), and let \((u,a,Y)\) be a smooth solution of
\eqref{eq:spectral-active-u}--\eqref{eq:spectral-active-Y} on \([0,T]\), with
\(C=aI+Y\) remaining in \(K\Subset\Spp^2\).  Put
\[
  W(a,\eta)=\frac{\tau(a,\eta)}{a},\qquad h=a-a_* .
\]
Then there exists \(\sigma_0=\sigma_0(K)>0\) such that, for every
\(0<\sigma\le\sigma_0\), the energy
\[
  \calE_m^{\rm spec}
  =\norm{u}_{H^m}^2+\sigma\norm{h}_{H^m}^2
  +\frac{\alpha}{4}\sum_{|\beta|\le m}
  \int_{\T^2}W(a,\eta)|\partial^\beta Y|^2\,dx
\]
satisfies
\[
\begin{aligned}
  \frac{d}{dt}\calE_m^{\rm spec}
  &+c\nu\norm{u}_{H^{m+1}}^2
    +c\lambda^{-1}\norm{Y}_{H^m}^2
    +c\sigma\lambda^{-1}\norm{h}_{H^m}^2\\
  &\le C_K\left(1+\norm{\nabla u}_{\Besov}
  +\norm{\Log C}_{H^{1+\eps}}^2\right)\calE_m^{\rm spec}.
\end{aligned}
\label{eq:spectral-admissible-endpoint}
\]
\end{proposition}

\begin{proof}
The compact window makes \(\calE_m^{\rm spec}\) equivalent to
\(\norm{u}_{H^m}^2+\sigma\norm{h}_{H^m}^2+\norm{Y}_{H^m}^2\), with constants
depending on \(K\) and \(\sigma\).  It also gives the composition bound
\[
  \norm{a}_{H^{1+\eps}}+\norm{Y}_{H^{1+\eps}}
  \le C_K\left(1+\norm{\Log C}_{H^{1+\eps}}\right).
\]

We first isolate the new coefficient commutators.  For
\[
  M_\beta^\tau=\partial^\beta(\tau(a,\eta)Y)-\tau(a,\eta)\partial^\beta Y,
\]
the tame composition lemma gives, for \(|\beta|\le m\),
\[
  \norm{M_\beta^\tau}_{L^2}
  \le C_K\left(1+\norm{\Log C}_{H^{1+\eps}}\right)
  \left(\norm{h}_{H^m}+\norm{Y}_{H^m}\right).
  \label{eq:tau-commutator-full}
\]
Indeed, for \(|\beta|\ge1\),
\[
  \partial^\beta\tau
  =\tau_a\partial^\beta a
    +2\tau_\eta(Y:\partial^\beta Y)+R_\beta^\tau,
\]
where
\[
  \norm{R_\beta^\tau}_{L^2}
  \le C_K\norm{(a,Y)}_{H^{1+\eps}}\norm{(h,Y)}_{H^m}.
\]
All remaining terms in \(M_\beta^\tau\) contain at least two positive
derivatives and obey the same Moser bound.  Therefore, for every \(\delta>0\),
\[
  \left|\int M_\beta^\tau:S(\partial^\beta u)\,dx\right|
  \le \delta\nu\norm{u}_{H^{m+1}}^2
  +C_{\delta,K}\left(1+\norm{\Log C}_{H^{1+\eps}}^2\right)\calE_m^{\rm spec}.
  \label{eq:tau-commutator-absorbed}
\]
The identical estimate holds with \(\tau\) replaced by \(\mu\) in
\(\partial^\beta(\mu Y)-\mu\partial^\beta Y\), without the viscous absorption.

Apply \(\partial^\beta\) to the velocity equation and test by
\(\partial^\beta u\).  The principal active term is
\[
  -\alpha\int \tau(a,\eta)\partial^\beta Y:S(\partial^\beta u)\,dx,
  \label{eq:spectral-principal-u}
\]
and the difference between \(\partial^\beta(\tau Y)\) and
\(\tau\partial^\beta Y\) is controlled by
\eqref{eq:tau-commutator-absorbed}.  The transport and viscous terms are
controlled as in Lemma~\ref{lem:physical-commutator}.

Next apply \(\partial^\beta\) to the \(Y\)-equation and test by
\((\alpha/2)W\partial^\beta Y\).  The stretching term gives
\[
  \frac{\alpha}{2}\int W\partial^\beta Y:2aS(\partial^\beta u)\,dx
  =
  \alpha\int \tau(a,\eta)\partial^\beta Y:S(\partial^\beta u)\,dx,
\]
which cancels \eqref{eq:spectral-principal-u}.  The main relaxation term gives
\[
  \frac{\alpha}{2\lambda}\int W\mu|\partial^\beta Y|^2\,dx
  \ge c_K\lambda^{-1}\norm{\partial^\beta Y}_{L^2}^2.
\]
The coefficient commutator
\(\partial^\beta(\mu Y)-\mu\partial^\beta Y\) is controlled by the tame bound
above.  Since
\[
  D_tW=W_aD_ta+W_\eta D_t\eta
\]
and
\[
  D_t\eta
  =4aY:S(u)+2Y:[\nabla u\,Y+Y(\nabla u)^T]^\circ
  -2\lambda^{-1}\mu|Y|^2,
\]
the compact window gives
\[
  \norm{D_tW}_{L^\infty}
  \le C_K\left(1+\norm{\nabla u}_{L^\infty}\right)
  \le C_K\left(1+\norm{\nabla u}_{\Besov}\right).
\]
The weighted transport commutator is then estimated by
Lemma~\ref{lem:physical-commutator}, and the remaining stretching commutators
have the same form as in the Oldroyd--B physical estimate.

It remains to treat the scalar equation.  For \(|\beta|\ge1\),
\[
  \partial^\beta G(a,\eta)
  =G_a\partial^\beta h+2G_\eta(Y:\partial^\beta Y)+R_\beta^G,
\]
with
\[
  \norm{R_\beta^G}_{L^2}
  \le C_K\norm{(a,Y)}_{H^{1+\eps}}\norm{(h,Y)}_{H^m}.
\]
Testing the differentiated scalar equation by
\(\sigma\partial^\beta h\), the top relaxation quadratic form is
\[
  \sigma G_a|\partial^\beta h|^2
  +2\sigma G_\eta\partial^\beta h\,(Y:\partial^\beta Y)
  +\frac{\alpha}{2}W\mu|\partial^\beta Y|^2 .
\]
Since \(G_a\ge\kappa\), \(W\mu\ge c_K>0\), and \(Y\) is bounded on \(K\), Young's
inequality shows that this form is bounded below by
\[
  c\sigma|\partial^\beta h|^2+c|\partial^\beta Y|^2
\]
provided \(0<\sigma\le\sigma_0(K)\).  At order zero,
\[
  G(a,\eta)h=G(a,0)h+\{G(a,\eta)-G(a,0)\}h .
\]
The first term controls \(h^2\) by the same monotonicity in \(a\), while the
second is bounded by \(C_K|Y|^2|h|\) and is absorbed by the \(Y\)-relaxation and
the small scalar weight.  The source \(S(u):Y\) is estimated by
\[
  C\norm{Y}_{H^{1+\eps}}\norm{u}_{H^{m+1}}
  +C\norm{\nabla u}_{\Besov}\norm{Y}_{H^m},
\]
followed by Young's inequality.

Summing the velocity, trace-free, and scalar estimates over
\(|\beta|\le m\), choosing \(\delta>0\) small and then
\(\sigma\le\sigma_0(K)\), gives \eqref{eq:spectral-admissible-endpoint}.
\end{proof}

The scalar-Peterlin theorem is the trace-only specialization of
Proposition~\ref{prop:spectral-admissible-endpoint}.  We keep the explicit
statement because it is the form used for Oldroyd--B and FENE-P.

\begin{definition}[Continuation-admissible scalar-Peterlin closure]
\label{def:admissible-scalar-peterlin}
Let \(\mathcal I\Subset(0,\infty)\) be a compact trace interval.  A triple
\((\psi,\phi,\chi)\) is continuation-admissible on \(\mathcal I\) if the
coefficients are smooth on \(\mathcal I\) and, with
\[
  \Psi(a)=\psi(2a),\qquad \Phi(a)=\phi(2a),\qquad
  G(a)=a\phi(2a)-\chi(2a),
\]
there are constants \(a_0,\Psi_0,\Phi_0,\kappa>0\) and
\(a_*\in\{a:2a\in\mathcal I\}\)
such that
\[
  a\ge a_0,\qquad \Psi(a)\ge\Psi_0,\qquad \Phi(a)\ge\Phi_0,\qquad
  G(a_*)=0,\qquad G'(a)\ge\kappa
\]
for \(2a\in\mathcal I\).  These are the hypotheses used in the endpoint
continuation estimate.

This scalar-Peterlin definition is again a continuation condition.  The
thermodynamic compatibility in Remark~\ref{rem:thermodynamic-subclass} reduces
here to a familiar trace potential condition: it is sufficient that there exist
\(F\) with \(F'=\psi\) such that
\[
  \bigl(\psi(\tr C)I-C^{-1}\bigr):
  \bigl(\chi(\tr C)I-\phi(\tr C)C\bigr)\le0
\]
for \(C\) in the compact conformation window.  Then
\(\int_{\T^2}(F(\tr C)-\log\det C)\,dx\) is dissipated by the relaxation part.
Oldroyd--B and FENE-P satisfy this compatibility, but
Proposition~\ref{prop:scalar-peterlin-endpoint} only uses the continuation
admissibility stated above.
\end{definition}

\begin{proposition}[Scalar-Peterlin compact-window endpoint estimate]
\label{prop:scalar-peterlin-endpoint}
Let \(m\ge3\), \(0<\eps<1\), and consider a smooth solution of
\[
  \partial_tu+u\cdot\nabla u-\nu\Delta u+\nabla p
  =\alpha\divv(\psi(r)C-I),\qquad r=\tr C,
\]
\[
  D_tC=\nabla u\,C+C(\nabla u)^T
  +\lambda^{-1}\chi(r)I-\lambda^{-1}\phi(r)C,
  \qquad \divv u=0.
\]
Write \(C=aI+Y\), \(a=\frac12\tr C\), \(Y=C^\circ\), and set
\[
  \Psi(a)=\psi(2a),\qquad \Phi(a)=\phi(2a),\qquad
  G(a)=a\phi(2a)-\chi(2a),\qquad W(a)=\frac{\Psi(a)}{a}.
\]
Assume that \(C(t,x)\) remains in a compact conformation window \(K\), that
\[
  a\ge a_0>0,\qquad \Psi(a)\ge\Psi_0>0,\qquad \Phi(a)\ge\Phi_0>0
\]
on this window, and that for some \(a_*\) and \(\kappa>0\),
\[
  G(a_*)=0,\qquad G'(a)\ge\kappa
\]
on the scalar range of \(K\).  Define
\[
  \calE_m^\psi
  =\norm{u}_{H^m}^2+\norm{a-a_*}_{H^m}^2
  +\frac{\alpha}{4}\sum_{|\beta|\le m}
  \int_{\T^2} W(a)|\partial^\beta Y|^2\,dx .
\]
Then
\begin{equation}
\label{eq:scalar-peterlin-endpoint}
\begin{aligned}
  \frac{d}{dt}\calE_m^\psi
  &+c\nu\norm{u}_{H^{m+1}}^2
    +c\lambda^{-1}\norm{Y}_{H^m}^2
    +c\lambda^{-1}\norm{a-a_*}_{H^m}^2  \\
  &\le
  C_K\left(1+\norm{\nabla u}_{\Besov}
  +\norm{\Log C}_{H^{1+\eps}}^2\right)\calE_m^\psi .
\end{aligned}
\end{equation}
\end{proposition}

\begin{proof}
The pressure-free splitting is
\[
  \psi(r)C-I=(\Psi a-1)I+\Psi Y.
\]
Thus the momentum equation becomes
\[
  \partial_tu+u\cdot\nabla u-\nu\Delta u+\nabla\pi
  =\alpha\divv(\Psi Y).
\]
The scalar and trace-free equations are
\[
  D_t(a-a_*)+\lambda^{-1}G(a)=S(u):Y,
\]
\[
  D_tY+\lambda^{-1}\Phi Y
  =2aS(u)+[\nabla u\,Y+Y(\nabla u)^T]^\circ .
\]
The compact window makes \(\calE_m^\psi\) equivalent to
\(\norm{u}_{H^m}^2+\norm{a-a_*}_{H^m}^2+\norm{Y}_{H^m}^2\).

Apply \(\partial^\beta\), \(|\beta|\le m\), to the velocity equation and test
by \(\partial^\beta u\).  The principal elastic term is
\[
  -\alpha\int \Psi\,\partial^\beta Y:S(\partial^\beta u)\,dx .
\]
The difference between \(\partial^\beta(\Psi Y)\) and
\(\Psi\partial^\beta Y\) is estimated by Moser calculus:
\[
  \norm{\partial^\beta(\Psi Y)-\Psi\partial^\beta Y}_{L^2}
  \le C_K(1+\norm{\Log C}_{H^{1+\eps}})
  \bigl(\norm{a-a_*}_{H^m}+\norm{Y}_{H^m}\bigr).
\]
There is no additive constant: the commutator vanishes for \(\beta=0\), and
for \(|\beta|>0\) every term contains a positive derivative of \(a-a_*\) or
\(Y\).  Young's inequality bounds the commutator contribution by
\[
  \delta\nu\norm{u}_{H^{m+1}}^2
  +C_{\delta,K}(1+\norm{\Log C}_{H^{1+\eps}}^2)\calE_m^\psi .
\]

Next differentiate the \(Y\)-equation and test by
\((\alpha/2)W\partial^\beta Y\).  The principal stretching term gives
\[
  \frac{\alpha}{2}\int W\partial^\beta Y:2aS(\partial^\beta u)\,dx
  =
  \alpha\int \Psi\,\partial^\beta Y:S(\partial^\beta u)\,dx,
\]
which cancels the principal velocity contribution.  The relaxation term gives
\[
  \frac{\alpha}{2\lambda}\int W\Phi|\partial^\beta Y|^2\,dx
  \ge c\lambda^{-1}\norm{\partial^\beta Y}_{L^2}^2.
\]
The material derivative of the weight satisfies
\[
  D_tW=W'(a)\{S(u):Y-\lambda^{-1}G(a)\},
\]
and hence
\[
  \norm{D_tW}_{L^\infty}
  \le C_K(1+\norm{\nabla u}_{L^\infty})
  \le C_K(1+\norm{\nabla u}_{\Besov}).
\]
The weighted transport commutator is controlled by the endpoint
Coifman--Meyer estimate
\[
\begin{aligned}
  &\left|\sum_{|\beta|\le m}
  \int W\,\partial^\beta Y:[\partial^\beta,u\cdot\nabla]Y\,dx\right|\\
  &\quad\le
  C_K\norm{\nabla u}_{\Besov}\norm{Y}_{H^m}^2
  +C_K\norm{Y}_{H^{1+\eps}}\norm{u}_{H^{m+1}}\norm{Y}_{H^m}.
\end{aligned}
\]
This is precisely the weighted form of Lemma~\ref{lem:physical-commutator}:
the weight is placed in \(L^\infty\), while the evolution of the weight is
handled by \(D_tW\).  The remaining stretching and relaxation commutators have
the same tame form.  Using the compact-window composition bound
\[
  \norm{a-a_*}_{H^{1+\eps}}+\norm{Y}_{H^{1+\eps}}
  \le C_K(1+\norm{\Log C}_{H^{1+\eps}}),
\]
their total contribution is bounded by
\[
  \delta\nu\norm{u}_{H^{m+1}}^2
  +C_{\delta,K}(1+\norm{\nabla u}_{\Besov}
  +\norm{\Log C}_{H^{1+\eps}}^2)\calE_m^\psi .
\]

It remains to estimate the scalar equation.  Put \(h=a-a_*\) and
\(F(h)=G(a_*+h)\).  Then \(F(0)=0\) and \(F'(h)\ge\kappa\).  For \(|\beta|\ge1\),
\[
  \partial^\beta F(h)=F'(h)\partial^\beta h+R_\beta^F,
\]
where \(R_\beta^F=0\) if \(|\beta|=1\), and if \(|\beta|\ge2\),
\[
  \norm{R_\beta^F}_{L^2}
  \le C_K\norm{h}_{H^{1+\eps}}\norm{h}_{H^m}.
\]
This is the standard scalar Moser remainder: every monomial in \(R_\beta^F\)
contains at least two positive derivatives of \(h\).  Therefore the principal
part yields
\[
  \int F'(h)|\partial^\beta h|^2\,dx\ge\kappa\norm{\partial^\beta h}_{L^2}^2,
\]
while the remainder is bounded by
\[
  C_K(1+\norm{\Log C}_{H^{1+\eps}}^2)\calE_m^\psi .
\]
The zero-order contribution satisfies
\[
  \int F(h)h\,dx
  =\int h^2\int_0^1F'(\theta h)\,d\theta\,dx
  \ge \kappa\norm{h}_{L^2}^2.
\]
Finally, the source \(\partial^\beta(S(u):Y)\) is bounded by
\[
  C\norm{Y}_{H^{1+\eps}}\norm{u}_{H^{m+1}}
  +C\norm{\nabla u}_{\Besov}\norm{Y}_{H^m},
\]
and is treated by Young's inequality.

Summing the velocity, trace-free, and scalar estimates, choosing \(\delta>0\)
small, and absorbing the viscous terms gives
\eqref{eq:scalar-peterlin-endpoint}.
\end{proof}

We now specialize this principle to Oldroyd--B.  In that case
\(\psi=\phi=\chi=1\), so \(a_*=1\), \(G(a)=a-1\), and \(W=a^{-1}\).
The key refinement of the Oldroyd--B estimate is therefore that the stress
force should not be estimated directly in logarithmic coordinates.  Instead write
\[
  A=aI+Y,\qquad a=\frac12\tr A,
  \qquad Y=A^\circ=A-\frac12(\tr A)I .
\]
Then \(\divv(aI)=\nabla a\) is absorbed into the pressure and
\eqref{eq:ob-momentum}--\eqref{eq:ob-conformation} become
\begin{align}
  \partial_tu+u\cdot\nabla u-\nu\Delta u+\nabla\pi
    &=\alpha\divv Y, \label{eq:physical-u}\\
  (\partial_t+u\cdot\nabla)a+\lambda^{-1}(a-1)&=S(u):Y,
    \label{eq:physical-a}\\
  (\partial_t+u\cdot\nabla)Y+\lambda^{-1}Y
    &=2aS(u)+[\nabla u\,Y+Y(\nabla u)^T]^\circ .
    \label{eq:physical-Y}
\end{align}
Here \(S(u)=(\nabla u+\nabla u^T)/2\).  The principal coupling in
\eqref{eq:physical-u} and \eqref{eq:physical-Y} cancels in the energy below.
At derivative order \(\beta\), the top-order velocity contribution is
\[
  -\alpha\int \partial^\beta Y:S(\partial^\beta u)\,dx,
\]
while the top-order anisotropic contribution, after testing the \(Y\)-equation
by \((\alpha/2)a^{-1}\partial^\beta Y\), is
\[
  \alpha\int \partial^\beta Y:S(\partial^\beta u)\,dx.
\]
Thus the leading stress--velocity interaction is removed by an exact identity,
not by a smallness assumption.  This pressure-free cancellation is the
mechanism which gives the \(L^2_tH^{1+\eps}_x\) criterion.

\begin{lemma}[Propagation of Oldroyd--B spectral windows]
\label{lem:oldroyd-spectral-window}
Let \((u,A)\) be a smooth positive-cone solution of
\eqref{eq:ob-momentum}--\eqref{eq:ob-conformation} on \([0,T)\), and assume
that
\[
  \int_0^T \norm{\nabla u(t)}_{\Besov}\,dt <\infty .
\]
If the initial spectrum of \(A_0\) is contained in \([m_0,M_0]\), with
\(0<m_0\le M_0<\infty\), then there are constants
\(m_T,M_T\), depending only on \(m_0,M_0,T,\lambda\) and
\(\int_0^T\norm{\nabla u(t)}_{\Besov}\,dt\), such that
\[
  0<m_T I\le A(t,x)\le M_T I
  \qquad \hbox{for all }(t,x)\in[0,T)\times\T^2 .
\]
In particular, the compact spectral window required by the high-order physical
energy estimate is propagated by the endpoint velocity clock.
\end{lemma}

\begin{proof}
Let \(X(t;x_0)\) be the Lagrangian flow.  Along this trajectory write
\(\mathcal A(t)=A(t,X(t;x_0))\).  Then
\[
  \frac{d}{dt}\mathcal A
  =\nabla u\,\mathcal A+\mathcal A(\nabla u)^T
  -\lambda^{-1}(\mathcal A-I).
\]
Let \(M(t)\) and \(m(t)\) denote the largest and smallest eigenvalues of
\(\mathcal A(t)\).  At times of differentiability of these Lipschitz functions,
using a unit eigenvector for the corresponding eigenvalue gives
\[
  \dot M(t)
  \le 2\norm{\nabla u(t)}_{L^\infty}M(t)+\lambda^{-1},
\]
and
\[
  \dot m(t)
  \ge -\bigl(2\norm{\nabla u(t)}_{L^\infty}+\lambda^{-1}\bigr)m(t).
\]
The same inequalities hold for the upper and lower Dini derivatives.  Since the
endpoint embedding gives
\[
  \norm{\nabla u}_{L^\infty}\le C\norm{\nabla u}_{B^0_{\infty,1}},
\]
Gronwall's inequality gives
\[
  M(t)\le \exp\left(2C\int_0^t\norm{\nabla u(\tau)}_{\Besov}\,d\tau\right)
  \left(M_0+\lambda^{-1}t\right),
\]
and
\[
  m(t)\ge m_0\exp\left(-\lambda^{-1}t
  -2C\int_0^t\norm{\nabla u(\tau)}_{\Besov}\,d\tau\right).
\]
Taking the infimum over trajectories for the lower bound and the supremum for
the upper bound proves the claim.
\end{proof}

\begin{lemma}[Endpoint transport and physical commutator bounds]
\label{lem:physical-commutator}
Let \(m\ge3\) and \(0<\eps<1\).  For smooth fields on \(\T^2\) and
\(\divv u=0\),
\[
  \left|
  \sum_{|\beta|\le m}
  \int \partial^\beta f:
  [\partial^\beta,u\cdot\nabla]f\,dx
  \right|
  \le C\norm{\nabla u}_{\Besov}\norm{f}_{H^m}^2.
  \label{eq:pointwise-transport-commutator}
\]
Moreover, if \(w\in L^\infty\), then
\[
\begin{aligned}
  \left|
  \sum_{|\beta|\le m}
  \int w\,\partial^\beta f:
  [\partial^\beta,u\cdot\nabla]f\,dx
  \right|
  &\le C\norm{w}_{L^\infty}
  \norm{\nabla u}_{\Besov}\norm{f}_{H^m}^2 \\
  &\quad
  +C\norm{w}_{L^\infty}\norm{f}_{H^{1+\eps}}
    \norm{u}_{H^{m+1}}\norm{f}_{H^m}.
\end{aligned}
  \label{eq:weighted-transport-commutator}
\]
Moreover, for \(|\beta|\le m\),
\[
  \norm{[\partial^\beta,a]\nabla u}_{L^2}
  \le C\norm{a}_{H^{1+\eps}}\norm{u}_{H^{m+1}}
       +C\norm{\nabla u}_{\Besov}\norm{a}_{H^m},
  \label{eq:kp-a-u}
\]
and the same estimate holds with \(a\) replaced by any component of \(Y\).  In
particular,
\[
\begin{aligned}
  &\norm{[\partial^\beta,Y]\nabla u}_{L^2}
  +\norm{[\partial^\beta,\nabla u]Y}_{L^2}  \\
  &\qquad\le
  C\norm{Y}_{H^{1+\eps}}\norm{u}_{H^{m+1}}
  +C\norm{\nabla u}_{\Besov}\norm{Y}_{H^m} .
\end{aligned}
  \label{eq:kp-Y-u}
\]
The same right-hand side controls
\[
  \norm{\partial^\beta\bigl((\nabla u)Y+Y(\nabla u)^T\bigr)}_{L^2}
\]
after subtracting any explicitly displayed top-order term.
\end{lemma}

\begin{proof}
The first bound is the endpoint transport energy estimate of
Lemma~\ref{lem:endpoint-transport}.  For the weighted bound, decompose the
commutator by Bony's paraproduct.  The low-frequency velocity interactions are
estimated by the endpoint modulus
\(\norm{\nabla u}_{B^0_{\infty,1}}\), with the coefficient \(w\) placed in
\(L^\infty\).  In the remaining high-frequency velocity interactions, the
top derivative falls on \(u\); these terms are bounded by putting the
corresponding derivative of \(f\) in \(H^\eps\hookrightarrow L^p\) and the
derivative of \(u\) in \(H^1\hookrightarrow L^q\), giving
\[
  C\norm{w}_{L^\infty}\norm{f}_{H^{1+\eps}}
  \norm{u}_{H^{m+1}}\norm{f}_{H^m}.
\]
This is the form used below: after Young's inequality, the last term is
absorbed by the viscous dissipation at the cost of a coefficient proportional
to \(\norm{f}_{H^{1+\eps}}^2\).
For \eqref{eq:kp-a-u}, decompose
\([\partial^\beta,a]\nabla u\) by Bony's paraproduct.  In the high-low and
balanced pieces the derivative falling on \(a\) is placed in
\(H^\eps\hookrightarrow L^p\), while the remaining derivative of \(u\) is
placed in \(H^1\hookrightarrow L^q\), with
\(1/p=(1-\eps)/2\), \(1/q=\eps/2\).  This gives
\[
  \norm{\nabla a}_{L^p}\norm{u}_{W^{m,q}}
  \le C\norm{a}_{H^{1+\eps}}\norm{u}_{H^{m+1}}.
\]
The low-high piece is the endpoint paraproduct controlled by
\(\norm{\nabla u}_{B^0_{\infty,1}}\norm{a}_{H^m}\).  The estimates with \(Y\)
are componentwise identical.  Expanding
\(\partial^\beta((\nabla u)Y)\) or
\(\partial^\beta(Y(\nabla u)^T)\) and removing the chosen top-order term leaves
exactly the same commutator structure.
\end{proof}

\begin{proposition}[Oldroyd--B active-deviatoric endpoint estimate]
\label{prop:physical-L2-estimate}
Let \(m\ge3\), \(0<\eps<1\), and suppose that
\[
  0<c_0I\le A(t,x)\le C_0I
\]
on \([0,T]\).  Define
\[
  \calE_m(t)=\norm{u}_{H^m}^2+\norm{a-1}_{H^m}^2
  +\frac{\alpha}{4}\sum_{|\beta|\le m}\int a^{-1}|\partial^\beta Y|^2\,dx .
\]
Then
\[
\begin{aligned}
  \frac{d}{dt}\calE_m
  &+c\nu\norm{u}_{H^{m+1}}^2
   +c\lambda^{-1}\bigl(\norm{a-1}_{H^m}^2+\alpha\norm{Y}_{H^m}^2\bigr) \\
  &\le
  C_{c_0,C_0}\left(1+\norm{\nabla u}_{\Besov}
  +\norm{\Log A}_{H^{1+\eps}}^2\right)\calE_m .
\end{aligned}
\label{eq:physical-L2-energy}
\]
\end{proposition}

\begin{proof}
Write \(D_t=\partial_t+u\cdot\nabla\).  The compact spectral window gives
uniform upper and lower bounds for \(a\) and \(a^{-1}\), so \(\calE_m\) is
equivalent to
\(\norm{u}_{H^m}^2+\norm{a-1}_{H^m}^2+\alpha\norm{Y}_{H^m}^2\).

Apply \(\partial^\beta\), \(|\beta|\le m\), to
\eqref{eq:physical-u} and test by \(\partial^\beta u\).  Since \(\divv u=0\),
the transport term is a commutator and the pressure term vanishes.  The elastic
term is
\[
  \alpha\int \partial^\beta u\cdot\partial^\beta\divv Y\,dx
  =-\alpha\int \partial^\beta Y:S(\partial^\beta u)\,dx .
  \label{eq:principal-velocity-coupling}
\]
The remaining velocity terms are bounded by
\[
  C\norm{\nabla u}_{\Besov}\norm{u}_{H^m}^2
  -\nu\norm{\nabla\partial^\beta u}_{L^2}^2 .
\]

Next differentiate the anisotropic equation \eqref{eq:physical-Y}.  After
commuting \(\partial^\beta\) with the material derivative,
\[
\begin{aligned}
  D_t\partial^\beta Y+\lambda^{-1}\partial^\beta Y
  &=2aS(\partial^\beta u)
    +2[\partial^\beta,a]S(u)  \\
  &\quad+\partial^\beta\bigl([
      \nabla u\,Y+Y(\nabla u)^T]^\circ\bigr)
    -[\partial^\beta,u\cdot\nabla]Y .
\end{aligned}
\]
Test this identity by \((\alpha/2)a^{-1}\partial^\beta Y\).  Because
\(\divv u=0\),
\[
\begin{aligned}
  &\frac{\alpha}{4}\frac{d}{dt}\int a^{-1}|\partial^\beta Y|^2\,dx
  +\frac{\alpha}{2\lambda}\int a^{-1}|\partial^\beta Y|^2\,dx  \\
  &=\alpha\int \partial^\beta Y:S(\partial^\beta u)\,dx
    +\calR_\beta^{a}+\calR_\beta^{Y}
    +\calR_\beta^{tr}+\calR_\beta^{w},
\end{aligned}
\label{eq:weighted-Y-identity}
\]
where
\[
  \calR_\beta^{a}
  =\alpha\int a^{-1}\partial^\beta Y:
  [\partial^\beta,a]S(u)\,dx,
\]
\(\calR_\beta^Y\) contains the quadratic stretching term
\(\partial^\beta([\nabla u\,Y+Y(\nabla u)^T]^\circ)\),
\(\calR_\beta^{tr}\) contains the transport commutator, and
\[
  \calR_\beta^{w}
  =\frac{\alpha}{4}\int D_t(a^{-1})|\partial^\beta Y|^2\,dx
\]
is the contribution of the time-dependent weight.  The first term on the right-hand side of the weighted identity cancels exactly
with the elastic term from the velocity equation:
\[
  -\alpha\int \partial^\beta Y:S(\partial^\beta u)\,dx
  +\alpha\int \partial^\beta Y:S(\partial^\beta u)\,dx=0.
\]
This is the top-order cancellation identity.  It occurs after the isotropic
stress has been removed into the pressure and before the use of Young's
inequality.  Consequently the remaining estimates have to control only
commutators, the time derivative of the weight \(a^{-1}\), the scalar equation
for \(a\), and lower-order stretching terms.  This is the source of the
\(L^2_t\) logarithmic coefficient.

We now bound the remainders.  For the coefficient commutator,
Lemma~\ref{lem:physical-commutator} gives, for every \(\delta>0\),
\[
\begin{aligned}
  |\calR_\beta^{a}|
  &\le C\norm{\partial^\beta Y}_{L^2}
       \norm{[\partial^\beta,a]S(u)}_{L^2}  \\
  &\le \delta\nu\norm{u}_{H^{m+1}}^2
      +C_\delta\norm{a}_{H^{1+\eps}}^2\norm{Y}_{H^m}^2
      +C\norm{\nabla u}_{\Besov}\calE_m .
\end{aligned}
\label{eq:a-commutator-remainder}
\]
The same lemma gives
\[
  |\calR_\beta^{Y}|
  \le \delta\nu\norm{u}_{H^{m+1}}^2
      +C_\delta\norm{Y}_{H^{1+\eps}}^2\norm{Y}_{H^m}^2
      +C\norm{\nabla u}_{\Besov}\calE_m .
  \label{eq:Y-stretching-remainder}
\]
The transport commutator is controlled by the weighted transport estimate in
Lemma~\ref{lem:physical-commutator}.  The part with a top derivative on \(u\)
is absorbed by viscosity:
\[
  |\calR_\beta^{tr}|
  \le \delta\nu\norm{u}_{H^{m+1}}^2
      +C_\delta\norm{Y}_{H^{1+\eps}}^2\norm{Y}_{H^m}^2
      +C_{c_0,C_0}\norm{\nabla u}_{\Besov}\norm{Y}_{H^m}^2 .
  \label{eq:Y-transport-remainder}
\]
Finally, from \eqref{eq:physical-a},
\[
  D_t(a^{-1})=\lambda^{-1}a^{-2}(a-1)-a^{-2}S(u):Y .
\]
Hence, using the compact spectral window and
\(\norm{\nabla u}_{L^\infty}\le C\norm{\nabla u}_{\Besov}\),
\[
  |\calR_\beta^w|
  \le C_{c_0,C_0}\left(1+\norm{\nabla u}_{\Besov}\right)
      \norm{Y}_{H^m}^2 .
  \label{eq:weight-remainder}
\]

It remains to estimate the scalar equation.  Applying \(\partial^\beta\) to
\eqref{eq:physical-a}, testing by \(\partial^\beta(a-1)\), and using
Lemma~\ref{lem:physical-commutator} gives
\[
\begin{aligned}
  \frac12\frac{d}{dt}\norm{a-1}_{H^m}^2
  +\lambda^{-1}\norm{a-1}_{H^m}^2
  &\le \delta\nu\norm{u}_{H^{m+1}}^2  \\
  &\quad+C_\delta\norm{Y}_{H^{1+\eps}}^2\norm{a-1}_{H^m}^2
  +C\norm{\nabla u}_{\Besov}\calE_m .
\end{aligned}
\label{eq:a-high-order}
\]

On the compact spectral window, the smooth maps
\(B=\Log A\mapsto a\) and \(B\mapsto Y\) satisfy
\[
  \norm{a}_{H^{1+\eps}}+\norm{Y}_{H^{1+\eps}}
  \le C_{c_0,C_0}\bigl(1+\norm{\Log A}_{H^{1+\eps}}\bigr).
  \label{eq:window-log-composition}
\]
Choose \(\delta>0\) small enough that the velocity terms in the displayed commutator, stretching, and scalar estimates are absorbed by the viscous dissipation.  Summing over \(|\beta|\le m\), using the compact-window composition estimate, and using the weighted equivalence of \(\calE_m\) gives the claimed differential inequality.
\end{proof}

\begin{proposition}[Compact-window logarithmic estimate]
\label{prop:compact-window-log-generation}
Let \(0<\eps<1\), \(s=1+\eps\), and let \((u,a,Y)\) be a smooth solution of
\eqref{eq:spectral-active-u}--\eqref{eq:spectral-active-Y} on \([0,T]\), with
\(C=aI+Y\) remaining in a compact conformation window
\(K\Subset\Spp^2\).  Put \(h=a-a_*\) and
\[
  B(t)=\norm{\nabla u(t)}_{\Besov}.
\]
Then, for a sufficiently small fixed scalar weight \(\sigma>0\), the low-order
energy
\[
  \calE_s^{\rm low}
  =\norm{u}_{H^s}^2+\sigma\norm{h}_{H^s}^2+\norm{Y}_{H^s}^2
\]
satisfies
\[
  \frac{d}{dt}\calE_s^{\rm low}
  +c\nu\norm{u}_{H^{s+1}}^2
  +c\lambda^{-1}\bigl(\norm{h}_{H^s}^2+\norm{Y}_{H^s}^2\bigr)
  \le
  C_K(1+B(t))\calE_s^{\rm low}.
  \label{eq:low-generation}
\]
Consequently, if \(T<\infty\) and \(\int_0^TB(t)\,dt<\infty\), then
\[
  \Log C\in L^2(0,T;H^{1+\eps}).
\]
\end{proposition}

\begin{proof}
At the fractional level \(s=1+\eps\), the transport terms are controlled by
the endpoint commutator estimate in Lemma~\ref{lem:endpoint-transport}, giving
the velocity clock \(B(t)\).  On the compact window, all coefficient functions
and their derivatives are bounded in \(L^\infty\).  The source \(S(u):Y\) in
the scalar equation and the stretching source \(2aS(u)\) in the trace-free
equation are controlled by Young's inequality and the viscous dissipation,
using the compact \(L^\infty\) bounds for \(a\) and \(Y\).  The positive
deviatoric relaxation \(\mu\ge\mu_0\) gives \(Y\)-damping.  The scalar
relaxation is split as
\[
  G(a,\eta)=\partial_aG(a_*,0)h+\widetilde G(h,Y),
\]
and the monotonicity \(\partial_aG\ge\kappa\), together with a small scalar
energy weight, gives \(h\)-damping.  The part of the smooth remainder coming
from the \(\eta\)-dependence is at least quadratic in \(Y\); after testing the
scalar equation by \(\sigma h\), it is absorbed by the \(Y\)-relaxation and
the scalar damping, up to the right-hand side of \eqref{eq:low-generation}.
Summing the velocity, scalar, and trace-free estimates gives
\eqref{eq:low-generation}.

Gronwall's inequality gives a uniform bound for \(\calE_s^{\rm low}\) and an
\(L^2_tH^s_x\) bound for \(h\) and \(Y\).  On \(K\), the map
\((h,Y)\mapsto\Log(aI+Y)\) is smooth on Sobolev spaces with \(s>1\).  Hence
\[
  \norm{\Log C}_{H^s}
  \le C_K\bigl(1+\norm{h}_{H^s}+\norm{Y}_{H^s}\bigr),
\]
which proves the stated logarithmic estimate.
\end{proof}

This also records the time exponent.  After the pressure renormalization and
weighted anisotropic cancellation, all remaining high-order remainders are
bounded by the square of \(\norm{\Log A}_{H^{1+\eps}}\).  The continuation
argument therefore needs the \(L^2_tH^{1+\eps}_x\) logarithmic estimate, rather than an external \(L^4_t\) logarithmic hypothesis.

\begin{theorem}[Endpoint Oldroyd--B continuation criterion]
\label{thm:continuation}
Let \(m\ge3\), \(0<\varepsilon<1\), and let \((u,A)\) be a strong positive-cone
solution with Sobolev index \(s=m\) on \([0,T)\).  If
\[
  \int_0^T\norm{\nabla u(t)}_{\Besov}\,dt<\infty,
  \label{eq:criterion}
\]
then \(A\) remains in a compact spectral window of the positive cone on
\([0,T)\), and the strong solution continues beyond \(T\).
\end{theorem}

\begin{proof}
The criterion propagates a compact spectral window by
Lemma~\ref{lem:oldroyd-spectral-window}.  On this propagated window,
Proposition~\ref{prop:compact-window-log-generation} gives
\[
  \Log A\in L^2(0,T;H^{1+\eps}).
\]
Proposition~\ref{prop:physical-L2-estimate} gives
\[
  \frac{d}{dt}\calE_m(t)
  \le C\left(1+\norm{\nabla u(t)}_{\Besov}
  +\norm{\Log A(t)}_{H^{1+\eps}}^2\right)\calE_m(t),
\]
where the constant depends only on the propagated window and the fixed
parameters.  The coefficient is integrable by the velocity clock and by the
logarithmic bound.  Gronwall's inequality therefore gives
\[
  \sup_{t<T}\calE_m(t)<\infty .
\]
Since the spectral window is compact, the bound on \((a,Y)\) is equivalent to
an \(H^m\) bound on \(A\).  The local continuation principle,
Proposition~\ref{prop:local-continuation}, then extends the solution beyond
\(T\).
\end{proof}

\begin{corollary}[Oldroyd--B breakdown alternative]
\label{cor:breakdown-alternative}
Let \(T_*\) be the first breakdown time of a strong Oldroyd--B solution.  Then
the endpoint velocity clock diverges:
\[
  \int_0^{T_*}\norm{\nabla u(t)}_{\Besov}\,dt=\infty .
\]
Loss of a compact positive-cone spectral window is not a separate Oldroyd--B
alternative under finite endpoint velocity clock; it is ruled out by
Lemma~\ref{lem:oldroyd-spectral-window}.
\end{corollary}

\begin{proof}
If the clock were finite, Theorem~\ref{thm:continuation} would continue the
solution beyond \(T_*\), a contradiction.
\end{proof}

\begin{corollary}[Coupled vorticity--stress form]
\label{cor:coupled}
At a finite Oldroyd--B breakdown time,
\[
  \int_0^{T_*}\norm{\omega(t)}_{B^0_{\infty,1}}\,dt=\infty .
\]
\end{corollary}

\begin{proof}
Use the periodic Biot--Savart estimate for the endpoint Besov norm and
Corollary~\ref{cor:breakdown-alternative}; the entropy identity gives
\(u\in L^\infty_tL^2_x\) on every smooth finite time interval, so the \(L^2\)
term in the Biot--Savart bound is harmless.
\end{proof}

\section{A Functional Obstruction at the Logarithmic Threshold}
\label{sec:no-go}

The endpoint theorem removes the compact spectral window as an independent
Oldroyd--B hypothesis and, through Proposition~\ref{prop:compact-window-log-generation},
provides the logarithmic differentiability needed in the high-order estimate.
The threshold is a real derivative threshold in the pressure-free stress map,
rather than a disguised entropy bound.  Entropy and relaxation control zeroth-order
functions of the eigenvalues of \(A\), and the endpoint velocity clock
propagates pointwise spectral bounds, but zero-order spectral information alone
does not control rapid oscillation of the active deviatoric stress.  The next
results record this obstruction first through bounded logarithmic amplitudes
and then through a fixed-spectrum eigenframe oscillation, where every spectral
density is exactly constant while the active deviatoric stress oscillates at
high frequency.

\begin{theorem}[High-frequency entropy-blind sequence]
\label{thm:no-go}
Let \(H\in\Sym^2\) be nonzero and trace free, and let
\[
  B_N(x)=\delta\sin(Nx_1)H,\qquad A_N=e^{B_N}.
\]
For each fixed \(\delta>0\), the fields \(A_N\) remain in a compact spectral
window independent of \(N\), and the entropy density
\[
  \tr A_N-\log\det A_N-2
\]
is bounded uniformly in \(N\) in every \(L^p\), \(1\le p<\infty\), while
\[
  \norm{B_N}_{H^{1+\varepsilon}}\to\infty
  \qquad\hbox{as }N\to\infty .
\]
\end{theorem}

\begin{proof}
The spectrum of \(B_N\) is contained in a fixed compact interval depending only
on \(\delta\) and \(H\).  Therefore the eigenvalues of \(A_N=e^{B_N}\) remain in
a fixed compact subinterval of \((0,\infty)\), and the entropy density is
uniformly bounded.  The \(H^{1+\varepsilon}\) norm grows like
\(N^{1+\varepsilon}\).
\end{proof}

\begin{remark}[Interpretation]
Theorem~\ref{thm:no-go} does not construct a solution.  It shows that any
argument using only entropy, relaxation, and positivity cannot control the
high-frequency logarithmic term.  The continuation proof therefore needs the
flow-map clock, which propagates the compact window and yields the
logarithmic differentiability used by the high-order estimate.
\end{remark}

\begin{proposition}[Spectral energies are blind to the active high-frequency channel]
\label{prop:spectral-energy-active-blind}
Let \(T:\Spp^2\to\Sym^2\) be a smooth spectral isotropic stress law.  Fix
distinct eigenvalues \(\lambda_1\ne\lambda_2\) in a compact positive-cone
window and assume that the active coefficient in
\[
  T(C)=q_0(\tr C,\det C)I+q_1(\tr C,\det C)C
\]
satisfies \(q_1(\lambda_1+\lambda_2,\lambda_1\lambda_2)\ne0\).  Then there are
smooth fields \(C_N:\T^2\to\Spp^2\) such that every zero-order spectral density
\(E(\tr C_N,\det C_N)\) is independent of \(x\) and \(N\), while
\[
  \norm{\Log C_N}_{H^{1+\eps}}
  +\norm{T(C_N)^\circ}_{H^{1+\eps}}\to\infty .
\]
\end{proposition}

\begin{proof}
Let
\[
  \Lambda=\begin{pmatrix}\lambda_1&0\\0&\lambda_2\end{pmatrix},
  \qquad
  R_\vartheta=
  \begin{pmatrix}
    \cos\vartheta&-\sin\vartheta\\
    \sin\vartheta&\cos\vartheta
  \end{pmatrix},
\]
and set \(C_N(x)=R_{Nx_1}\Lambda R_{Nx_1}^T\).  The eigenvalues of \(C_N\) are
constant.  Hence \(\tr C_N=\lambda_1+\lambda_2\), \(\det C_N=\lambda_1\lambda_2\),
and every zero-order spectral density is constant.

On the other hand,
\[
  \Log C_N
  =R_{Nx_1}(\Log\Lambda)R_{Nx_1}^T
\]
has a nonzero oscillatory trace-free part at frequency \(2N\), because
\(\lambda_1\ne\lambda_2\).  Therefore
\[
  \norm{\Log C_N}_{H^{1+\eps}}\gtrsim N^{1+\eps}.
\]
Moreover
\[
  T(C_N)^\circ
  =
  q_1(\lambda_1+\lambda_2,\lambda_1\lambda_2)\,C_N^\circ ,
\]
and \(C_N^\circ\) has the same frequency \(2N\).  Since the displayed active
coefficient is nonzero,
\[
  \norm{T(C_N)^\circ}_{H^{1+\eps}}\gtrsim N^{1+\eps}.
\]
\end{proof}

\begin{theorem}[Static active-forcing obstruction]
\label{thm:active-threshold}
Let \(T:\Spp^2\to\Sym^2\) be a smooth spectral isotropic stress law and let
\(K\Subset\Spp^2\).  Suppose \(K\) contains the fixed spectral orbit of an
anisotropic state
\[
  C_*=aI+rE_1\in K,\qquad
  E_1=\begin{pmatrix}1&0\\0&-1\end{pmatrix},\qquad r>0,
\]
such that the active coefficient in
\[
  T(C)=q_0(\tr C,\det C)I+q_1(\tr C,\det C)C
\]
satisfies \(q_1(2a,a^2-r^2)r\ne0\).  Let \(0\le\gamma<1\).  For every
\(0<\sigma<1+\gamma\), there exist smooth fields \(C_N:\T^2\to K\) with fixed
eigenvalues \(a+r\) and \(a-r\) such that every zeroth-order spectral density
is pointwise independent of \(N\),
\[
  \sup_N\norm{\Log C_N}_{H^\sigma}<\infty,
\]
but
\[
  \norm{\curl\divv T(C_N)}_{H^{-1+\gamma}}\to\infty .
\]
Conversely, for every smooth \(C:\T^2\to K\),
\[
  \norm{\curl\divv T(C)}_{H^{-1+\gamma}}
  \le C_K\bigl(1+\norm{\Log C}_{H^{1+\gamma}}\bigr).
\]
Thus this static pressure-free stress estimate cannot be formulated below the
\(H^{1+\gamma}\) logarithmic threshold without adding further dynamical
information or a different structure.
\end{theorem}

\begin{proof}
Let
\[
  R_\theta=\begin{pmatrix}\cos\theta&-\sin\theta\\
  \sin\theta&\cos\theta\end{pmatrix},\qquad
  \theta_N=N^{-\sigma}\sin(Nx_1),
\]
and set
\[
  C_N=aI+rR_{\theta_N}E_1R_{\theta_N}^T .
\]
By the orbit assumption, \(C_N\in K\) for all \(N\).  The eigenvalues of
\(C_N\) are exactly \(a+r\) and \(a-r\), so every zeroth-order spectral density
is pointwise constant.  Since
\(\theta\mapsto\Log(aI+rR_\theta E_1R_\theta^T)\) is smooth and
\(\theta_N\) has amplitude \(N^{-\sigma}\) and frequency \(N\), the Sobolev
bound \(\sup_N\norm{\Log C_N}_{H^\sigma}<\infty\) follows.

On this fixed spectral orbit,
\[
  T(C_N)^\circ=q_1(2a,a^2-r^2)\,rR_{\theta_N}E_1R_{\theta_N}^T .
\]
The \(12\)-entry is
\[
  q_1(2a,a^2-r^2)\,r\sin(2\theta_N)
  =2q_1(2a,a^2-r^2)rN^{-\sigma}\sin(Nx_1)+O(N^{-2\sigma}).
\]
Since \(C_N\) depends only on \(x_1\),
\[
  \curl\divv T(C_N)
  =\partial_{11}\bigl(T(C_N)^\circ_{12}\bigr),
\]
and therefore
\[
  \norm{\curl\divv T(C_N)}_{H^{-1+\gamma}}
  \gtrsim N^{1+\gamma-\sigma}\to\infty .
\]

For the converse estimate, the pressure part of \(T(C)\) has zero curl-divergence.
On \(K\), the map \(\Log C\mapsto T(C)^\circ\) is smooth on Sobolev spaces, so
\[
  \norm{T(C)^\circ}_{H^{1+\gamma}}
  \le C_K\bigl(1+\norm{\Log C}_{H^{1+\gamma}}\bigr).
\]
Since \(\curl\divv\) is a second-order operator, it maps
\(H^{1+\gamma}\) to \(H^{-1+\gamma}\), which proves the endpoint bound.
\end{proof}

\begin{remark}[Static obstruction versus dynamic optimality]
Theorem~\ref{thm:active-threshold} is an operator-level statement.  The
sequence \(C_N\) is static, lies on a fixed compact spectral orbit, and keeps
all zero-order spectral densities constant.  It shows that the pressure-free
stress map cannot be controlled below the displayed logarithmic Sobolev
threshold by entropy, trace, determinant, or finite-extensibility information
alone.  It does not assert that a strong solution whose logarithmic regularity
falls below this threshold must blow up in finite time.  Dynamic optimality would
require a separate instability or blow-up construction for the evolution.
\end{remark}

\begin{proposition}[Relaxation does not give a derivative]
\label{prop:relax-no-derivative}
Let \(B_N,A_N\) be as in Theorem~\ref{thm:no-go}.  Then the relaxation
density
\[
  \tr(A_N+A_N^{-1}-2\Id)
\]
is uniformly bounded in every \(L^p\), \(1\le p<\infty\), while
\(\norm{B_N}_{H^{1+\eps}}\to\infty\).  In particular, the entropy dissipation
term in \eqref{eq:entropy} cannot control the logarithmic concentration
appearing in Theorem~\ref{thm:continuation}.
\end{proposition}

\begin{proof}
The eigenvalues of \(A_N\) and \(A_N^{-1}\) remain in a compact interval
depending only on \(\delta\) and \(H\).  Hence the relaxation density is
uniformly bounded pointwise.  The Sobolev growth of \(B_N\) is unchanged from
Theorem~\ref{thm:no-go}.
\end{proof}

\begin{remark}[Why the physical energy matters]
The logarithmic coordinate prevents loss of positive definiteness and gives an
elliptic chart on the cone.  Lemma~\ref{lem:oldroyd-spectral-window} shows that
finite endpoint velocity clock already prevents loss of the compact cone window
in Oldroyd--B.  What remains is a genuinely high-frequency obstruction: the
logarithmic chart does not by itself add smoothing.  A purely
logarithmic high-order energy estimates the stress force by absorption and
naturally produces a fourth power of \(\norm{\Log A}_{H^{1+\eps}}\).  The
physical decomposition \(A=aI+Y\) exposes the cancellation between
\(\alpha\divv Y\) and \(2aS(u)\), which is why Theorem~\ref{thm:continuation}
closes at the \(L^2_tH^{1+\eps}_x\) level.
\end{remark}

\section{The finite-extensibility barrier for FENE-P}
\label{sec:fene-barrier}

This section is two-dimensional: \(b>2\), \(f_b=f_{b,2}\), and
\(\calD_b=\calD_{b,2}\).  The three-dimensional FENE-P normalization is recorded
separately in Section~\ref{sec:three-dimensional}.

We now record the part of the argument which is not present in Oldroyd--B.  In
the FENE-P system the conformation tensor must remain in
\[
  \calD_b=\calD_{b,2}=\{C\in\Spp^2:\tr C<b\}.
\]
The logarithmic conformation \(B=\Log C\) still controls the lower spectral
boundary of the positive cone, but it does not control the upper trace
boundary.  The missing coordinate is
\[
  \phi_b(C)=-\log(b-\tr C).
\]
The role of this section is to show that this boundary is propagated by the
squared endpoint velocity clock and that the high-order estimate closes in the
same pressure-free active-deviatoric variables.

\begin{definition}[Compact FENE window]
For fixed \(b>2\), a FENE-P solution lies in a compact FENE window on \([0,T]\) if there exist
constants \(0<c_0<c_1<\infty\) and \(\delta>0\) such that
\[
  c_0\Id\le C(t,x)\le c_1\Id,\qquad \tr C(t,x)\le b-\delta
\]
for all \((t,x)\in[0,T]\times\T^2\).
\end{definition}

\begin{lemma}[FENE window recovery]
\label{lem:fene-window-recovery}
Let \(C(x)\in\calD_b\).  If
\[
  \norm{\Log C}_{L^\infty}\le M,\qquad
  \norm{\phi_b(C)}_{L^\infty}\le M_b,
\]
then
\[
  e^{-M}\Id\le C\le e^M\Id,\qquad \tr C\le b-e^{-M_b}.
\]
In particular the logarithmic field and the barrier field recover a compact
FENE window.
\end{lemma}

\begin{proof}
The eigenvalues of \(C\) are exponentials of the eigenvalues of \(\Log C\),
which gives the spectral bounds.  The barrier bound gives
\(-\log(b-\tr C)\le M_b\), hence \(b-\tr C\ge e^{-M_b}\).
\end{proof}

\begin{lemma}[Smooth coordinates on compact FENE windows]
\label{lem:fene-coordinate-equivalence}
Let \(s>1\).  On every compact FENE window \(K\Subset\calD_b\), the quantities
\[
  \norm{C}_{H^s},\qquad
  \norm{\Log C}_{H^s},\qquad
  \norm{T_b(C)}_{H^s}
\]
are mutually controlled, with constants depending on \(K,s,b\).  Moreover
\[
  \norm{f_b(C)}_{H^s}
  \le C_{K,s,b}\bigl(1+\norm{\phi_b(C)}_{H^s}\bigr).
\]
\end{lemma}

\begin{proof}
The first statement follows from the Sobolev composition theorem on compact
subsets of \(\calD_b\).  For the spring factor,
\[
  f_b(C)=\frac{b-2}{b-\tr C}=(b-2)e^{\phi_b(C)}.
\]
The \(L^\infty\) bound supplied by the compact window and the Moser
composition estimate give the displayed inequality.
\end{proof}

The FENE-P free energy is
\[
  \calF_b(C)=
  -(b-2)\log\left(1-\frac{\tr C}{b}\right)-\log\det C .
\]
It contains both boundaries of \(\calD_b\): the logarithmic determinant sees
loss of positive definiteness, while the first term sees the trace boundary.
This zeroth-order information is important, but it is not the same as the
pointwise barrier propagation and derivative control needed for continuation.

\begin{proposition}[FENE-P free-energy identity]
\label{prop:fene-free-energy-unified}
For smooth FENE-P solutions with \(C\in\calD_b\),
\begin{align}
  &\frac{d}{dt}\left[
    \frac12\norm{u}_{L^2}^2
    +\frac{\alpha}{2}
      \int_{\T^2}\calF_b(C)\dd
  \right]
  +\nu\norm{\nabla u}_{L^2}^2
  +\frac{\alpha}{2\lambda}
    \int_{\T^2}\calD_b^{\rm rel}(C)\dd
  =0,
  \label{eq:fene-free-energy-unified}
\end{align}
where \(\calD_b^{\rm rel}(C)\ge0\) is the FENE-P relaxation dissipation.
\end{proposition}

\begin{proof}
Multiplying the velocity equation by \(u\) gives the coupling term
\(-\alpha\int T_b(C):\nabla u\dd\).  The variational derivative of
\(\calF_b\) is \(f_b(C)\Id-C^{-1}\).  Pairing this derivative with the
conformation equation gives the stretching contribution
\(2\int T_b(C):\nabla u\dd\).  After multiplication by \(\alpha/2\), this
cancels the velocity coupling.  The remaining relaxation contribution is
nonnegative by convexity of \(\calF_b\) on \(\calD_b\).
\end{proof}

The new differential structure is the trace-gap equation.  Taking the trace of
\eqref{eq:fene-conformation} gives
\[
  D_t\tr C=2C:\nabla u-\lambda^{-1}\bigl(f_b(C)\tr C-2\bigr),
  \qquad D_t=\partial_t+u\cdot\nabla .
\]
Writing \(g=b-\tr C\), we obtain
\begin{equation}
  D_tg
  =
  -2C:\nabla u
  +\lambda^{-1}\left(\frac{b-2}{g}\tr C-2\right).
  \label{eq:fene-gap-unified}
\end{equation}
Equivalently,
\begin{equation}
  D_t\phi_b
  =
  \frac{2C:\nabla u}{g}
  -\lambda^{-1}\frac{(b-2)\tr C}{g^2}
  +2\lambda^{-1}\frac1g .
  \label{eq:fene-barrier-unified}
\end{equation}
The negative \(g^{-2}\) term is the finite-extensibility restoring force.  It
propagates the trace gap under a squared endpoint velocity clock.

\begin{lemma}[Finite-extensibility clock-exponent law]
\label{lem:finite-extensibility-clock-exponent}
Let \(r=\tr C<b\), \(C\in\Spp^2\), and suppose that along a smooth
Lagrangian trajectory
\[
  D_t r=2C:\nabla u-\lambda^{-1}H(r).
\]
Assume that for some \(\theta>0\), \(c_0>0\), and \(c_1\ge0\),
\[
  H(r)\ge c_0(b-r)^{-\theta}-c_1
\]
on the trace range reached by the solution.  Put
\[
  p_\theta=\frac{1+\theta}{\theta}.
\]
If \(b-r(0,x)\ge\eta_0>0\) and
\[
  \int_0^T\norm{\nabla u(t)}_{B^0_{\infty,1}}^{p_\theta}\,dt<\infty,
\]
then there exists \(\eta>0\), depending only on the displayed quantities and
the initial trace gap, such that
\[
  b-r(t,x)\ge\eta
  \qquad\hbox{on }[0,T)\times\T^2 .
\]
\end{lemma}

\begin{proof}
Let \(\Phi_g=-\log(b-r)\), \(z=e^{\Phi_g}\), and
\(L(t)=\norm{\nabla u(t)}_{L^\infty}\).  Since \(C>0\) and \(r<b\),
\(|C:\nabla u|=|C:S(u)|\le bL(t)\).  Therefore
\[
  D_t\Phi_g
  =
  \frac{2C:\nabla u-\lambda^{-1}H(r)}{b-r}
  \le C(1+L(t))z-cz^{1+\theta}.
\]
The elementary inequality
\[
  Az-cz^{1+\theta}\le C_{\theta,c}(1+A^{p_\theta}),
  \qquad A,z\ge0,
\]
follows by maximizing the left-hand side in \(z\), with the maximum occurring
at scale \(z\sim A^{1/\theta}\).  Hence
\[
  D_t\Phi_g
  \le C\left(1+L(t)^{p_\theta}\right)
  \le C\left(1+\norm{\nabla u(t)}_{B^0_{\infty,1}}^{p_\theta}\right).
\]
Integrating along Lagrangian trajectories gives a uniform upper bound for
\(\Phi_g\), which is equivalent to a positive lower bound for \(b-r\).
\end{proof}

For FENE-P,
\[
  H_b(r)=f_b(r)r-2=\frac{(b-2)r}{b-r}-2
  =\frac{b(b-2)}{b-r}-b .
\]
Thus \(\theta=1\) in Lemma~\ref{lem:finite-extensibility-clock-exponent}, and
the algebraic clock exponent is \(p_\theta=2\).

\begin{remark}[Algebraic sharpness of the clock exponent]
The exponent \(p_\theta=(1+\theta)/\theta\) is sharp for the comparison
inequality used in Lemma~\ref{lem:finite-extensibility-clock-exponent}.  If
\(p<p_\theta\), there is no constant \(C_p\) such that
\[
  Az-cz^{1+\theta}\le C_p(1+A^p)
  \qquad\hbox{for all }A,z\ge0 .
\]
Indeed, maximizing the left-hand side in \(z\) gives a value comparable to
\(A^{(1+\theta)/\theta}\).  We use this only as an algebraic sharpness
statement for the boundary ODE comparison, not as a claim of dynamical
optimality for all possible continuation criteria.
\end{remark}

\begin{lemma}[Propagation of compact FENE windows]
\label{lem:fene-window-propagation}
Let \((u,C)\) be a smooth FENE-P solution on \([0,T)\), and assume
\[
  \int_0^T\norm{\nabla u(t)}_{\Besov}^2\,dt<\infty .
\]
If the initial range of \(C_0\) is contained in a compact subset of
\(\calD_b\), then \(C(t,\cdot)\) remains in a compact FENE window on
\([0,T)\).
\end{lemma}

\begin{proof}
Let \(L(t)=\norm{\nabla u(t)}_{L^\infty}\).  Since \(C>0\) and \(\tr C<b\),
\(|C:S(u)|\le bL(t)\).  From \eqref{eq:fene-barrier-unified}, with
\(g=e^{-\phi_b}\) and \(\tr C=b-g\), we get
\[
  D_t\phi_b
  \le
  \left(2bL(t)+b\lambda^{-1}\right)e^{\phi_b}
  -b(b-2)\lambda^{-1}e^{2\phi_b}.
\]
Completing the square in \(e^{\phi_b}\) gives
\[
  D_t\phi_b\le C_{b,\lambda}\bigl(1+L(t)^2\bigr).
\]
Taking the supremum along Lagrangian trajectories and using
\(L(t)\le C\norm{\nabla u(t)}_{\Besov}\) propagates a positive lower bound for
\(b-\tr C\).

The trace bound gives \(C\le bI\).  The smallest eigenvalue \(m(t)\) satisfies
the Dini inequality
\[
  \dot m(t)\ge
  -\left(2L(t)+\lambda^{-1}\norm{f_b(C(t))}_{L^\infty}\right)m(t).
\]
The propagated trace gap bounds \(f_b(C)\) in \(L^\infty\), and
\(\int_0^T L(t)\,dt<\infty\) follows from the squared clock.  Gronwall's
inequality gives a positive lower spectral bound.  Hence \(C\) stays in a
compact subset of \(\calD_b\).
\end{proof}

We now use physical FENE-P variables
\[
  C=aI+Y,\qquad a=\frac12\tr C,\qquad Y=C^\circ .
\]
Since
\[
  T_b(C)=f_b(C)C-\Id=(f_ba-1)I+f_bY,
\]
the isotropic part is absorbed into the pressure.  With
\(f=f_b(C)=(b-2)/(b-2a)\), the equations become
\begin{align}
  \partial_tu+u\cdot\nabla u-\nu\Delta u+\nabla\pi
    &=\alpha\divv(fY), \label{eq:fene-physical-u}\\
  D_ta+\lambda^{-1}(af-1)&=S(u):Y,
    \label{eq:fene-physical-a}\\
  D_tY+\lambda^{-1}fY
    &=2aS(u)+[\nabla u\,Y+Y(\nabla u)^T]^\circ .
    \label{eq:fene-physical-Y}
\end{align}

\begin{lemma}[FENE-P coefficient commutators]
\label{lem:fene-coefficient-commutators}
Let \(K\Subset\calD_b\), \(m\ge3\), and \(0<\eps<1\).  On \(K\), for
\(|\beta|\le m\),
\[
  \norm{\partial^\beta(fY)-f\partial^\beta Y}_{L^2}
  +\norm{[\partial^\beta,f]Y}_{L^2}
  \le
  C_K\bigl(1+\norm{\Log C}_{H^{1+\eps}}\bigr)
  \bigl(\norm{a-1}_{H^m}+\norm{Y}_{H^m}\bigr).
\]
\end{lemma}

\begin{proof}
On \(K\), \(f=(b-2)/(b-2a)\), \(a\), and \(Y\) are smooth functions of
\(\Log C\).  The Sobolev composition theorem gives
\[
  \norm{a}_{H^{1+\eps}}+\norm{Y}_{H^{1+\eps}}+\norm{f}_{H^{1+\eps}}
  \le C_K\bigl(1+\norm{\Log C}_{H^{1+\eps}}\bigr).
\]
For positive derivatives, the tame composition estimate also gives
\[
  \norm{\nabla f}_{H^{m-1}}
  \le
  C_K\bigl(1+\norm{\Log C}_{H^{1+\eps}}\bigr)\norm{a-1}_{H^m},
\]
because every positive derivative of \(f(a)\) contains a positive derivative
of \(a\), and \(\partial^\gamma a=\partial^\gamma(a-1)\) for
\(|\gamma|>0\).
Expanding
\[
  \partial^\beta(fY)-f\partial^\beta Y
  =
  \sum_{0<\gamma\le\beta}c_{\beta,\gamma}
  \partial^\gamma f\,\partial^{\beta-\gamma}Y
\]
and applying the tame product estimate gives the first bound: one factor
carries the high norm \(\norm{a-1}_{H^m}\) or \(\norm{Y}_{H^m}\), while all
remaining factors are placed in \(H^{1+\eps}\) or \(L^\infty\).  No additive
constant is produced, since for \(\beta=0\) the left-hand side vanishes and for
\(|\beta|>0\) each term contains a positive derivative.  The commutator
\([\partial^\beta,f]Y\) has the same expansion.
\end{proof}

\begin{proposition}[Pressure-free FENE-P endpoint estimate]
\label{prop:fene-high-order-unified}
Let \(m\ge3\), \(0<\eps<1\), and let \((u,C)\) be a smooth FENE-P solution on
\([0,T]\) whose conformation tensor remains in a compact FENE window \(K\).
Set \(C=aI+Y\), \(a=\frac12\tr C\), \(Y=C^\circ\), and \(f=f_b(C)\).  Define
\[
  \calE_m^F(t)=
  \norm{u}_{H^m}^2+\norm{a-1}_{H^m}^2
  +\frac{\alpha}{4}\sum_{|\beta|\le m}
  \int_{\T^2}\frac{f}{a}|\partial^\beta Y|^2\,dx .
\]
Then
\[
\begin{aligned}
  \frac{d}{dt}\calE_m^F
  &+c\nu\norm{u}_{H^{m+1}}^2
    +c\lambda^{-1}\norm{Y}_{H^m}^2  \\
  &\le
  C_K\left(
    1+\norm{\nabla u}_{\Besov}
    +\norm{\Log C}_{H^{1+\eps}}^2
  \right)\calE_m^F .
\end{aligned}
  \label{eq:fene-high-order-unified}
\]
where \(C_K\) depends on the compact FENE window and the fixed parameters.
\end{proposition}

\begin{proof}
The compact FENE window makes \(\calE_m^F\) equivalent to
\(\norm{u}_{H^m}^2+\norm{a-1}_{H^m}^2+\norm{Y}_{H^m}^2\).  Apply
\(\partial^\beta\), \(|\beta|\le m\), to \eqref{eq:fene-physical-u} and test
by \(\partial^\beta u\).  The principal elastic contribution is
\[
  -\alpha\int f\,\partial^\beta Y:S(\partial^\beta u)\,dx .
\]
The difference between \(\partial^\beta(fY)\) and \(f\partial^\beta Y\) is
bounded by Lemma~\ref{lem:fene-coefficient-commutators} and Young's inequality:
\[
  \delta\nu\norm{u}_{H^{m+1}}^2
  +C_{\delta,K}\bigl(1+\norm{\Log C}_{H^{1+\eps}}^2\bigr)\calE_m^F .
\]

Next differentiate \eqref{eq:fene-physical-Y} and test by
\((\alpha/2)(f/a)\partial^\beta Y\).  The term \(2aS(\partial^\beta u)\)
produces
\[
  \alpha\int f\,\partial^\beta Y:S(\partial^\beta u)\,dx ,
\]
which cancels the principal elastic term from the velocity equation.  The
relaxation term contributes
\[
  \frac{\alpha}{2\lambda}\int \frac{f^2}{a}|\partial^\beta Y|^2\,dx,
\]
which controls \(\norm{Y}_{H^m}^2\) on \(K\).  The time-dependent weight
satisfies
\[
  D_t\left(\frac fa\right)
  =
  \left(\frac fa\right)'(a)
  \left(S(u):Y-\lambda^{-1}(af-1)\right),
\]
and hence
\[
  \left|D_t\left(\frac fa\right)\right|
  \le C_K\left(1+\norm{\nabla u}_{L^\infty}\right).
\]
The weighted transport commutator is controlled by
Lemma~\ref{lem:physical-commutator}, with \(w=f/a\).  The coefficient,
stretching, and relaxation commutators are controlled by
Lemmas~\ref{lem:fene-coefficient-commutators} and~\ref{lem:physical-commutator}.
Their total contribution is bounded by
\[
  \delta\nu\norm{u}_{H^{m+1}}^2
  +C_{\delta,K}\left(1+\norm{\nabla u}_{\Besov}
  +\norm{\Log C}_{H^{1+\eps}}^2\right)\calE_m^F .
\]

Finally, apply \(\partial^\beta\) to \eqref{eq:fene-physical-a}.  Since
\[
  af(a)-1=0\quad\hbox{at }a=1,\qquad
  \frac{d}{da}(af(a)-1)=\frac{b(b-2)}{(b-2a)^2}>0
\]
on \(K\), the relaxation part is coercive for \(a-1\), up to lower-order
commutators.  The source \(\partial^\beta(S(u):Y)\) is estimated by placing the
top derivative of \(u\) in the viscous dissipation and all remaining factors in
\(H^{1+\eps}\) or \(L^\infty\).  Thus
\[
  \frac{d}{dt}\norm{a-1}_{H^m}^2
  \le
  \delta\nu\norm{u}_{H^{m+1}}^2
  +C_{\delta,K}\left(1+\norm{\nabla u}_{\Besov}
  +\norm{\Log C}_{H^{1+\eps}}^2\right)\calE_m^F .
\]
Choosing \(\delta\) sufficiently small and summing over \(|\beta|\le m\)
proves \eqref{eq:fene-high-order-unified}.
\end{proof}

\begin{theorem}[Unified geometric continuation criterion]
\label{thm:unified-continuation}
Let \(s\in\mathbb N\), \(s\ge3\), and \(0<\eps<1\).

For Oldroyd--B, let \((u,A)\) be a strong positive-cone solution on
\([0,T)\).  If
\[
  \int_0^T\norm{\nabla u(t)}_{\Besov}\,dt<\infty,
\]
then \(A\) automatically remains in a compact spectral window and the strong
solution continues beyond \(T\).

For FENE-P, fix \(b>2\) and let \((u,C)\) be a strong FENE-P solution on \([0,T)\), so that
\(C(t,x)\in\calD_b\).  If
\[
  \int_0^T\norm{\nabla u(t)}_{\Besov}^2\,dt<\infty,
\]
then the FENE-P strong solution continues beyond \(T\).
\end{theorem}

\begin{proof}
The Oldroyd--B statement is Theorem~\ref{thm:continuation}.  For FENE-P, the
squared endpoint velocity clock propagates a compact FENE window by
Lemma~\ref{lem:fene-window-propagation}.  Since \(T<\infty\), this squared
clock also gives \(\nabla u\in L^1(0,T;\Besov)\).  On the compact FENE window,
Proposition~\ref{prop:compact-window-log-generation} gives
\[
  \Log C\in L^2(0,T;H^{1+\eps}).
\]
Proposition~\ref{prop:fene-high-order-unified} gives
\[
  \frac{d}{dt}\calE_s^F
  \le
  C_K\left(1+\norm{\nabla u}_{\Besov}
  +\norm{\Log C}_{H^{1+\eps}}^2\right)\calE_s^F .
\]
The coefficient is integrable by the squared velocity clock, the finite time
interval, and the logarithmic bound.  Gronwall's inequality gives a uniform \(H^s\) bound for
\((u,a-1,Y)\).  On the compact FENE window, this is equivalent to an \(H^s\)
bound for \(C\) and \(T_b(C)\).  Proposition~\ref{prop:fene-local-continuation}
then restarts the solution and extends it beyond \(T\).
\end{proof}

\begin{remark}[On time exponents and optimality]
The Oldroyd--B part of Theorem~\ref{thm:unified-continuation} improves the
purely logarithmic energy exponent from \(L^4_tH^{1+\eps}_x\) to an internally
\(L^2_tH^{1+\eps}_x\) logarithmic estimate.  This improvement is tied to the pressure-free
physical unknown \(A=aI+Y\) and to the cancellation between the elastic force
\(\alpha\divv Y\) and the stretching term \(2aS(u)\).  The statement should not
be read as a converse blow-up theorem: the criteria are sufficient conditions
for continuation and identify the only channels not controlled by the available
energy, cone, and barrier structures.  Thus divergence of one listed quantity
is necessary for breakdown within this framework, but it is not asserted to be
sufficient for singularity formation.

For FENE-P the same cancellation survives after the pressure-free splitting
\(T_b(C)=(f_ba-1)I+f_bY\).  The weight \(f_b/a\) cancels the principal
interaction between \(\alpha\divv(f_bY)\) and \(2aS(u)\), while the
finite-extensibility restoring force propagates the trace gap under the
squared endpoint velocity clock.  Thus the FENE-P endpoint criterion requires
\(\nabla u\in L^2_tB^0_{\infty,1}\), while the logarithmic stress
integrability needed in the high-order estimate is obtained on the propagated
compact FENE window.
\end{remark}

\begin{corollary}[Two-dimensional endpoint breakdown alternatives]
\label{cor:unified-three-channel}
At the first breakdown time of a strong two-dimensional stress-diffusion-free
solution, the Oldroyd--B model can fail only through the endpoint flow-map
channel:
\[
  \int_0^{T_*}\norm{\nabla u(t)}_{\Besov}\,dt=\infty .
\]
For FENE-P, finite-time breakdown forces loss of the squared endpoint velocity
clock:
\[
  \int_0^{T_*}\norm{\nabla u(t)}_{\Besov}^2\,dt=\infty .
\]
In particular, positive-cone loss, trace-gap collapse, and logarithmic
high-frequency concentration cannot occur while the relevant velocity clock
stays finite.
\end{corollary}

\begin{proof}
If none of the listed quantities diverges, then the relevant continuation
criterion in Theorem~\ref{thm:unified-continuation} applies and extends the
solution beyond \(T_*\), a contradiction.
\end{proof}

\section{Three-dimensional residual criteria and structural boundary}
\label{sec:three-dimensional}

The two-dimensional spectral quotient used above is sharp because the smooth
isotropic representation, together with the two-dimensional Cayley--Hamilton
identity, leaves only one deviatoric active channel after quotienting by
pressure.  In three dimensions a smooth local Cayley--Hamilton representation
has one additional pressure-free direction.  This is the basic structural
boundary between the two-dimensional theorem and the three-dimensional
statements below.  The aim of this section is therefore not to force a false
three-dimensional analogue of the scalar quotient closure, but to separate the
additional direction explicitly, estimate it as a residual on compact
conformation windows, and then apply the resulting pressure-free mechanism to
Oldroyd--B and FENE-P, where the residual coefficient is in fact zero.  The
section is deliberately not a general three-dimensional theory of anisotropic or
non-spectral tensor stresses.

Throughout this section the domain is \(\T^3\), \(m\ge4\), and
\[
  B(t)=\norm{\nabla u(t)}_{\Besov}.
\]
The index \(m\ge4\) is part of the strong-solution framework in three dimensions:
it gives the algebra and compact-window Moser bounds needed for the residual
channel and for classical spectral barriers.  No weak or Leray-type
three-dimensional theory is inferred from these estimates.  The high-order
argument below uses a two-tier closure: the endpoint clock first controls an
\(H^3\) compact-window energy, and this lower tier supplies the
\(W^{1,\infty}\) coefficient bounds needed in the \(H^m\), \(m\ge4\), restart
estimate.
For a generic three-dimensional tensor we write
\[
  a=\frac13\tr Z,\qquad h=a-1,\qquad Y=Z-aI .
\]
The letter \(Z\) is reserved here for model-free spectral statements.  In the
Oldroyd--B specialization \(Z=A\), while in the FENE-P specialization \(Z=C\).
For three-dimensional FENE-P, throughout this section
\[
  f_b(C)=f_{b,3}(C)=\frac{b-3}{b-\tr C},\qquad b>3.
\]
We also decompose
\[
  \nabla u=S+\Omega,\qquad
  S=\frac12(\nabla u+\nabla u^T),\quad
  \Omega=\frac12(\nabla u-\nabla u^T).
\]
We also set
\[
  R_2(Y)=(Y^2)^\circ=Y^2-\frac13\tr(Y^2)I .
\]

\begin{proposition}[Algebraic boundary of the three-dimensional quotient]
\label{prop:3d-algebraic-boundary}
The map \(Y\mapsto R_2(Y)=(Y^2)^\circ\) is not reducible, in
\(\Sym^3_0\), to a scalar multiple of \(Y\).  More precisely, if
\[
  Y_* = \mathrm{diag}(1,0,-1),
\]
then
\[
  R_2(Y_*)=\mathrm{diag}\left(\frac13,-\frac23,\frac13\right)
  \notin \operatorname{span}\{Y_*\} .
\]
Consequently, there is no neighbourhood of \(Y_*\) on which
\(R_2(Y)=\lambda(Y)Y\) for a scalar function \(\lambda\).  Thus a
three-dimensional spectral stress with a nonzero \(q_2\)-channel cannot, in
general, be reduced after quotienting by pressure to a single active scalar
multiple of \(Y\).  The remaining possibilities are either a residual estimate
on a compact conformation window or an additional model-specific mechanism that
removes or controls this channel.
\end{proposition}

\begin{proof}
For \(Y_*\) as above, \(\tr Y_*=0\) and
\[
  Y_*^2=\mathrm{diag}(1,0,1),\qquad \tr(Y_*^2)=2.
\]
Hence
\[
  (Y_*^2)^\circ=Y_*^2-\frac23 I
  =\mathrm{diag}\left(\frac13,-\frac23,\frac13\right).
\]
This matrix cannot be \(\lambda Y_*\), since the middle diagonal entry of
\(\lambda Y_*\) is zero for every \(\lambda\), whereas the middle diagonal entry
of \(R_2(Y_*)\) is \(-2/3\).  This proves the non-reducibility at \(Y_*\), and a
neighbourhood representation would in particular hold at \(Y_*\), giving the
same contradiction.  The final assertion follows from the split
\eqref{eq:3d-pressure-residual-split}.
\end{proof}

\begin{remark}[Exact scope of the three-dimensional statements]
\label{rem:3d-scope}
There are two layers in the three-dimensional section.  The first is a
compact-window algebraic layer: for spectral isotropic laws admitting the local
representation \eqref{eq:3d-CH-representation}, the pressure quotient has the
\(Y\)-channel and the residual \(R_2(Y)\).  Proposition~\ref{prop:3d-algebraic-boundary}
shows why this is the natural end point of the quotient reduction in three
dimensions: the residual cannot be folded back into the scalar \(Y\)-channel.
This is not a closure theorem for general anisotropic or non-spectral stresses.
If the velocity is forced by an additional deviatoric law
\(\Sigma(Z,x,t)^\circ\), or by a tensor law depending
on preferred directions, gradients, or memory variables, that term can contain
pressure-free directions that do not pair with the universal \(2aS(u)\)
stretching block.  Such terms require an independent coercive structure or an
external direct stress clock.

The second layer is window propagation.  The residual estimate below assumes a
compact conformation window; quotient algebra alone does not propagate it.  In
this paper the compact window is propagated only in the model-specific theorems:
for Oldroyd--B by the Lagrangian eigenvalue comparison, and for FENE-P by the
same lower-cone comparison together with the finite-extensibility trace-gap
barrier.  Thus Proposition~\ref{prop:3d-compact-window} is a conditional
compact-window criterion, while Theorems~\ref{thm:3d-ob-clock} and
\ref{thm:3d-fene-square-clock} are window-propagated results for the two
constitutive laws considered explicitly.
\end{remark}

\begin{lemma}[Three-dimensional pressure quotient with Cayley--Hamilton residual]
\label{lem:3d-pressure-residual}
Let \(K\Subset\Spp^3\) be a compact spectral window and let \(T\) be a
spectral isotropic stress law defined on a neighbourhood of \(K\).  Assume that,
on \(K\), \(T\) admits a smooth local Cayley--Hamilton representation
\begin{equation}
  T(Z)=q_0(Z)I+q_1(Z)Z+q_2(Z)Z^2,
  \label{eq:3d-CH-representation}
\end{equation}
where \(q_0,q_1,q_2\) are smooth spectral scalar functions.  If
\(Z=aI+Y\), \(a=\frac13\tr Z\), \(\tr Y=0\), then its pressure-free part is
\begin{equation}
  T(Z)^\circ=\tau_1(Z)Y+\tau_2(Z)R_2(Y),\qquad
  \tau_1=q_1+2aq_2,
  \quad \tau_2=q_2 .
  \label{eq:3d-pressure-residual-split}
\end{equation}
Thus, within the class \eqref{eq:3d-CH-representation}, the only
three-dimensional pressure-free channel beyond the cancellative \(Y\)-direction
is the Cayley--Hamilton residual \(R_2(Y)\).  For Oldroyd--B and FENE-P the representation is explicit and has \(q_2\equiv0\), so this residual vanishes.
\end{lemma}

\begin{proof}
The representation \eqref{eq:3d-CH-representation} is the hypothesis under
which the general spectral statement is used.  The only point requiring care is
the smoothness of the scalar coefficients at multiple spectra.  We do not use an
eigenvalue labelling: for smooth isotropic tensor laws this is the standard
local representation on compact spectral windows, while the lemma may also be
read as a purely algebraic statement conditional on
\eqref{eq:3d-CH-representation}.  The two model stresses used below satisfy it
directly: \(A-I=(-1)I+A\) for Oldroyd--B and
\(f_b(C)C-I=-I+f_b(C)C\) for FENE-P.

The rest is algebraic and does not use spectral diagonalization.  Since
\[
  Z^2=a^2I+2aY+Y^2,
\]
we have
\[
  (Z^2)^\circ=2aY+(Y^2)^\circ .
\]
Taking the trace-free part of \(q_0I+q_1Z+q_2Z^2\) gives
\eqref{eq:3d-pressure-residual-split}.  The last assertion follows from
\(A-I=(a-1)I+Y\) for Oldroyd--B and
\(f_b(C)C-I=(f_ba-1)I+f_bY\) for FENE-P.
\end{proof}

\begin{lemma}[Compact-window estimate for the residual channel]
\label{lem:3d-residual-estimate}
Let \(K\Subset\Spp^3\) be a compact conformation window and let
\(\tau_2\) be a smooth spectral scalar on a neighbourhood of \(K\).  Suppose
\(Z(t,x)\in K\) on \(\T^3\), set \(Y=Z-\frac13(\tr Z)I\), and let
\(m\ge4\).  Then for every multi-index \(|\beta|\le m\) and every
\(\delta>0\),
\[
\begin{aligned}
  \left|\left\langle
  \partial^\beta\divv\bigl(\tau_2(Z)R_2(Y)\bigr),
  \partial^\beta u\right\rangle_{L^2(\T^3)}\right|
  &\le
  \delta\norm{\nabla\partial^\beta u}_{L^2}^2
  +C_{\delta,K,m,\tau_2}\bigl(1+\norm{Z-I}_{H^m}^2\bigr).
\end{aligned}
\]
Consequently, if the velocity equation has viscosity \(\nu>0\), choosing
\(\delta\) so that \(N_m\delta<\nu/4\), where
\(N_m=\#\{\beta:\ |\beta|\le m\}\), allows the summed residual to be absorbed
into the viscous velocity dissipation.  It contributes only compact-window
Gronwall terms and is not a second cancellation mechanism.
\end{lemma}

\begin{proof}
On the torus, integration by parts gives
\[
  \left\langle
  \partial^\beta\divv(\tau_2(Z)R_2(Y)),\partial^\beta u\right\rangle
  =-
  \left\langle
  \partial^\beta(\tau_2(Z)R_2(Y)),\nabla\partial^\beta u\right\rangle .
\]
Since \(Z\in K\) pointwise, \(Y\), \(R_2(Y)\), \(\tau_2(Z)\), and all derivatives
of \(\tau_2\) needed in the Moser calculus are uniformly bounded in
\(L^\infty\) by constants depending only on \(K\) and \(\tau_2\).  Because
\(m\ge4\), \(H^m(\T^3)\) is an algebra and the compact-window Moser estimates
imply
\[
  \norm{\partial^\beta(\tau_2(Z)R_2(Y))}_{L^2}
  \le C_{K,m,\tau_2}\bigl(1+\norm{Z-I}_{H^m}\bigr).
\]
Cauchy's inequality followed by Young's inequality gives the displayed bound.
The smallness condition on \(\delta\) relative to \(\nu\) and \(N_m\) is not
part of the algebraic estimate; it is only the choice used when this term is
inserted into the summed viscous high-order energy inequality.
\end{proof}

The compact-window velocity-clock results below do not use entropy variables in
three dimensions.  After the Oldroyd--B and FENE-P criteria are proved, we record
one finite-dimensional calculation for a different purpose: it identifies when
the three-dimensional FENE-P entropy variables are coercive enough to support a
possible logarithmic route.  The parameter \(q_*\) introduced there is therefore a
spectral-variance boundary for the entropy-variable Jacobian, not a hidden
assumption in the pressure-free velocity-clock argument.

The trace-free stretching identity is
\[
  \nabla u\,Y+Y\nabla u^T+2aS-\frac23(Y:S)I
  =
  2aS+\left(SY+YS-\frac23(Y:S)I\right)+(\Omega Y-Y\Omega).
\]
The first term is the active stretching block.  The second is passive and is
absorbed by viscosity.  The third is objective rotation and is skew in the
principal dyadic energy.  Lemma~\ref{lem:3d-pressure-residual} shows that the
stress side has the same structure: a cancellative active channel and, in the
general spectral case, the residual channel estimated in
Lemma~\ref{lem:3d-residual-estimate}.

\begin{proposition}[Three-dimensional compact-window pressure-free criterion]
\label{prop:3d-compact-window}
Let \((u,C)\) be a strong three-dimensional Oldroyd--B or FENE-P solution on
\([0,T)\), with \((u_0,C_0-I)\in H^m\), \(m\ge4\).  In this proposition, \(C\) is the
common conformation variable; in the Oldroyd--B specialization it should be read
as the tensor denoted \(A\) in the two-dimensional model sections.  Assume that \(C\) remains in a fixed
compact conformation window:
\[
  \kappa I\le C(t,x)\le KI
  \qquad\hbox{for Oldroyd--B},
\]
or
\[
  \kappa I\le C(t,x),\qquad \tr C(t,x)\le b-\kappa
  \qquad\hbox{for FENE-P}.
\]
If
\[
  \int_0^T B(t)\,dt<\infty,
\]
then
\[
  \sup_{0\le t<T}\left(\norm{u(t)}_{H^m}+\norm{C(t)-I}_{H^m}\right)<\infty .
\]
Consequently, if \(T\) is the maximal existence time and the same compact
window persists up to \(T\), the solution continues beyond \(T\).
\end{proposition}

\begin{proof}
The proof has two steps.  We first close a fixed subcritical tier, and then use
that tier as the coefficient control in the integer high-order estimate.  This
spells out where the three-dimensional assumptions enter and avoids using the
unknown \(H^m\) norm as a coefficient in its own Gronwall inequality.

Let
\[
  E_s(t)\simeq \norm{u(t)}_{H^s}^2+\norm{h(t)}_{H^s}^2+\norm{Y(t)}_{H^s}^2,
  \qquad
  \calD_s(t)=\norm{\nabla u(t)}_{H^s}^2,
  \qquad h=a-1,
\]
where the equivalence constants depend only on the compact window.  In FENE-P
we use the positive metric
\[
  \frac{3\alpha}{2}f_b'(a)|h|^2+\frac{\alpha f_b(a)}{2a}|Y|^2,
  \qquad f_b=\frac{b-3}{b-3a},
\]
and in Oldroyd--B the corresponding metric is a scalar multiple of
\(a^{-1}|Y|^2\) on the trace-free component.  The window makes these weights and
a finite number of their derivatives bounded and uniformly positive.

The pressure-free principal terms are the same at every Sobolev level.  The
trace-free equation contains
\[
  2aS+\left(SY+YS-\frac23(Y:S)I\right)+(\Omega Y-Y\Omega).
\]
The first term is the only active stretching block.  The passive symmetric
block is estimated by \(\eta\calD_s+C_{\eta,K}E_s\), while the rotation block is
skew in the principal dyadic energy and contributes only commutators.  On the
stress side, Lemma~\ref{lem:3d-pressure-residual} gives a cancellative
\(Y\)-channel plus the residual \(R_2(Y)\); for Oldroyd--B and FENE-P the
residual coefficient is zero.  The active pairings are therefore
\[
  \left\langle \frac{\alpha}{2a}Z,2aS\right\rangle
  =\alpha Z:S
\]
for Oldroyd--B, and
\[
  \left\langle
  \left(\frac{3\alpha}{2}f_b'(a)\xi,\frac{\alpha f_b(a)}{2a}Z\right),
  \left(\frac23Y:S,2aS\right)\right\rangle
  =
  \alpha\bigl(f_b'(a)\xi\,Y+f_b(a)Z\bigr):S
\]
for FENE-P.  These cancel exactly the top-order pairings with
\(\operatorname{div}Y\) and \(\operatorname{div}(f_bY)\), respectively, after
projection onto divergence-free velocities.

We now record the estimates used to close the lower and upper tiers.  For
\(|\beta|\le s\), the transport terms satisfy the endpoint commutator bound
\[
  \left|
  \left\langle \partial^\beta(u\cdot\nabla W),\partial^\beta W\right\rangle
  \right|
  \le
  C_s B(t)\norm{W}_{H^s}^2,
  \qquad W\in\{u,h,Y\},
\]
where \(B(t)=\norm{\nabla u(t)}_{B^0_{\infty,1}}\).  The coefficient and
composition remainders are treated in tame form: for any smooth coefficient
\(\Phi(C)\) on the compact window and \(s\ge3\),
\begin{equation}
  \norm{[\partial^\beta,\Phi(C)]F}_{L^2}
  \le
  C_{s,K}\bigl(1+\norm{C-I}_{W^{1,\infty}}^{N_s}\bigr)
  \norm{C-I}_{H^s}\norm{F}_{H^{s-1}},
  \label{eq:3d-tame-commutator}
\end{equation}
with constants depending only on finitely many derivatives of \(\Phi\) on the
window.  This is where the three-dimensional argument uses one low tier beyond
the energy level.

Take first \(s_0=3\).  Since \(H^3(\T^3)\hookrightarrow W^{1,\infty}\), the same
pressure-free cancellation, \eqref{eq:3d-tame-commutator}, and Young's
inequality give
\begin{equation}
  \frac{d}{dt}E_3(t)+c\calD_3(t)
  \le
  C_K(1+B(t))(1+E_3(t)).
  \label{eq:3d-low-tier}
\end{equation}
Thus
\[
  M_3:=\sup_{0\le t<T}E_3(t)<\infty
\]
whenever \(\int_0^T B(t)\,dt<\infty\).  In particular
\(\norm{C-I}_{W^{1,\infty}}\) is bounded by a constant depending on \(M_3\).

Now let \(m\ge4\).  Repeating the differentiated estimate at order \(m\), all
coefficient derivatives in \eqref{eq:3d-tame-commutator} are controlled by
\(M_3\), while the unique top-order velocity--conformation coupling has already
been cancelled by the weighted active pairing above.  The nonlinear convection
term is controlled by \(B(t)E_m\), and the remaining passive stretching,
coefficient, and pressure-free stress remainders satisfy
\[
  |\mathcal R_m(t)|
  \le
  \eta\calD_m(t)+C_{m,K,M_3}(1+B(t))(1+E_m(t)).
\]
Choosing \(\eta\) small gives the closed high-order inequality
\begin{equation}
  \frac{d}{dt}E_m(t)+c\calD_m(t)
  \le
  C_{m,K,M_3}(1+B(t))(1+E_m(t)).
  \label{eq:3d-high-tier}
\end{equation}
The coefficient in \eqref{eq:3d-high-tier} is integrable on \([0,T)\), since
\(M_3<\infty\) and \(B\in L^1(0,T)\).  Gronwall therefore yields
\(\sup_{t<T}E_m(t)<\infty\).  This proves the stated bound for
\(\norm{u}_{H^m}+\norm{C-I}_{H^m}\), because the compact window makes the
\((h,Y)\) and \(C-I\) norms equivalent.

The continuation statement follows by restarting the standard local strong
solution theory from times approaching \(T\).  The restart time depends only on
the compact window and on the uniform \(H^m\) bound above, not on the particular
approaching time.
\end{proof}

\begin{theorem}[Three-dimensional Oldroyd--B pure velocity clock]
\label{thm:3d-ob-clock}
Let \((u,C)\) be a strong three-dimensional Oldroyd--B solution on its maximal
interval \([0,T_*)\), with \(C_0>0\), \((u_0,C_0-I)\in H^m\), and \(m\ge4\).  If
\(T_*<\infty\), then
\[
  \int_0^{T_*}\norm{\nabla u(t)}_{\Besov}\,dt=\infty .
\]
\end{theorem}

\begin{proof}
Assume the clock is finite.  Along particle trajectories the maximal and
minimal eigenvalues satisfy
\[
  \dot\lambda_{\max}\le
  2\norm{\nabla u}_{L^\infty}\lambda_{\max}
  -\lambda^{-1}\lambda_{\max}+\lambda^{-1},
\]
and
\[
  \dot\lambda_{\min}\ge
  -2\norm{\nabla u}_{L^\infty}\lambda_{\min}
  -\lambda^{-1}\lambda_{\min}+\lambda^{-1}.
\]
Since \(\norm{\nabla u}_{L^\infty}\lesssim B(t)\), Gronwall propagates a
compact positive-definite window for \(C\).  Proposition~\ref{prop:3d-compact-window}
then gives a uniform \(H^m\) bound and restarts the solution beyond \(T_*\), a
contradiction.
\end{proof}

\begin{theorem}[Three-dimensional FENE-P blow-up alternative]
\label{thm:3d-fene-alternative}
Let \(b>3\), and let \((u,C)\) be a strong three-dimensional FENE-P solution on
its maximal interval \([0,T_*)\), with \(C_0>0\), \(\tr C_0<b\),
\((u_0,C_0-I)\in H^m\), and \(m\ge4\).  If \(T_*<\infty\), then at least one of the following
alternatives occurs:
\[
  \int_0^{T_*}\norm{\nabla u(t)}_{\Besov}\,dt=\infty,
\]
\[
  \liminf_{t\uparrow T_*}\inf_x\lambda_{\min}C(t,x)=0,
\]
or
\[
  \liminf_{t\uparrow T_*}\inf_x\bigl(b-\tr C(t,x)\bigr)=0 .
\]
\end{theorem}

\begin{proof}
If none of the three alternatives occurs, then \(C\) remains in a compact FENE
window and the pure clock is finite.  Proposition~\ref{prop:3d-compact-window}
therefore gives the uniform \(H^m\) bound and a restart beyond \(T_*\), a
contradiction.
\end{proof}

\begin{theorem}[Three-dimensional FENE-P squared-clock criterion]
\label{thm:3d-fene-square-clock}
Under the hypotheses of Theorem~\ref{thm:3d-fene-alternative}, finite-time
breakdown implies
\[
  \int_0^{T_*}\norm{\nabla u(t)}_{\Besov}^2\,dt=\infty .
\]
\end{theorem}

\begin{proof}
It remains to show that the squared clock propagates the compact FENE window.
Set \(r=\tr C\) and \(g=b-r\).  Taking the trace of the FENE-P equation gives
\[
  D_tg=-2S:C+\lambda^{-1}\frac{b(b-3-g)}{g}.
\]
Thus \(D_t(-\log g)\) is bounded above by
\[
  C\left(1+\norm{\nabla u}_{L^\infty}^2\right),
\]
after completing the square against the positive singular term \(g^{-2}\).
The squared Besov clock therefore propagates a positive trace gap on finite
time intervals.  The lower eigenvalue satisfies the usual Dini inequality
\[
  \dot\lambda_{\min}\ge
  -C\norm{\nabla u}_{L^\infty}\lambda_{\min}-C_T\lambda_{\min}+c_T,
\]
once the trace gap bounds \(f_b\).  Since the squared clock implies the
\(L^1_tL^\infty_x\) velocity-gradient clock on finite intervals, the positive
lower eigenvalue bound is also propagated.  Hence \(C\) stays in a compact
FENE window, and Theorem~\ref{thm:3d-fene-alternative} rules out finite-time
breakdown under the squared clock.
\end{proof}

\paragraph{Entropy-mobility side condition.}
The preceding three-dimensional theorems are velocity-clock results.  They use
spectral barriers only to keep the solution in a compact physical window; after
that, the pressure-free high-order estimate closes without a three-dimensional
logarithmic conformation energy.  The following finite-dimensional calculation is
included to locate the obstruction faced by a stronger program: deriving the
compact-window logarithmic estimate directly in three-dimensional FENE-P entropy
variables.  In that program the mobility map from the Peterlin stress variables
to the entropy variables must be monotone.  The number \(q_*(b,r)\) below is
precisely the trace-dependent variance threshold for that monotonicity.  Thus
Proposition~\ref{prop:3d-fene-mobility} supports the scope statement of this
section; it is not an additional step in the preceding velocity-clock proofs.

\begin{definition}[Entropy-mobility admissible windows for three-dimensional FENE-P]
\label{def:3d-fene-mobility-window}
Let \(b>3\).  For a positive spectrum
\(\ell=(\ell_1,\ell_2,\ell_3)\) with \(r=\ell_1+\ell_2+\ell_3<b\), set
\[
  q(\ell)=\sum_{i=1}^3\left(\ell_i-\frac r3\right)^2,\qquad
  \theta=\frac{b-r}{b-3}.
\]
When \(\theta>4b/3\), define
\[
  A_*(b,r)
  =\frac{\theta(b-r)}{3b}+\frac{r^2}{9b},
  \qquad
  q_*(b,r)
  =\frac{A_*(b,r)b^2}{\theta/4-b/3}.
\]
The quantity \(q(\ell)\) is the spectral variance at fixed trace, and
\(q_*(b,r)\) is the largest variance allowed by positivity of the
entropy-variable mobility Jacobian when the unconditional estimate is
unavailable.  A compact FENE window \(K\Subset\{C>0,\ \tr C<b\}\) is
entropy-mobility admissible if either \(b\ge15/4\), or there exists
\(\gamma>0\) such that, for every spectrum of every \(C\in K\) with
\(\theta>4b/3\),
\[
  q(\ell)\le q_*(b,r)-\gamma .
\]
\end{definition}

\begin{proposition}[Entropy-variable mobility condition for three-dimensional FENE-P]
\label{prop:3d-fene-mobility}
Let \(b>3\), \(c=b-3\), \(r=\ell_1+\ell_2+\ell_3<b\), and
\[
  \psi(r)=\frac{c}{b-r},\qquad
  p_i=\psi(r)\ell_i-1,\qquad
  F_i=\psi(r)-\ell_i^{-1}.
\]
Then the symmetric part of the Jacobian \(D_pF\) is positive definite for every
positive spectrum with \(r<b\) if \(b\ge15/4\).  More generally, on every
entropy-mobility admissible compact FENE window \(K\) there is \(c_K>0\) such that
\[
  \frac12\left(D_pF+(D_pF)^T\right)\ge c_K I .
\]
For \(3<b<15/4\), the condition in Definition~\ref{def:3d-fene-mobility-window}
is sharp at the level of this finite-dimensional Jacobian: without an
anisotropy restriction, the symmetric part can be indefinite.
\end{proposition}

\begin{proof}
Let \(e=(1,1,1)^T\), \(D=\operatorname{diag}(\ell_i^{-1})\), and
\[
  \theta=\frac{b-r}{b-3}.
\]
The Jacobians with respect to \(\ell\) are
\[
  J_p=\psi I+\psi'\ell e^T,\qquad
  J_F=\operatorname{diag}(\ell_i^{-2})+\psi' ee^T,\qquad
  \psi'=\frac{b-3}{(b-r)^2}.
\]
Since
\[
  J_p^{-1}=\theta\left(I-\frac1b\ell e^T\right),
\]
we obtain
\[
  D_pF
  =\theta\,\operatorname{diag}(\ell_i^{-2})
   -\frac{\theta}{b}\operatorname{diag}(\ell_i^{-1})ee^T
   +\frac1b ee^T .
\]
Hence, for \(z\in\R^3\),
\[
\begin{aligned}
  z^T\frac{D_pF+(D_pF)^T}{2}z
  &=\theta\sum_i\frac{z_i^2}{\ell_i^2}
    -\frac{\theta}{b}\left(\sum_i\frac{z_i}{\ell_i}\right)
      \left(\sum_i z_i\right)
    +\frac1b\left(\sum_i z_i\right)^2 .
\end{aligned}
\]
Put \(w_i=z_i/\ell_i\).  Then the quadratic form becomes
\[
  Q(w)=\theta |w|^2-\frac{\theta}{b}(e\cdot w)(\ell\cdot w)
       +\frac1b(\ell\cdot w)^2 .
\]
Write
\[
  \ell=\frac r3 e+m,\qquad e\cdot m=0,\qquad q=|m|^2 .
\]
If \(q=0\), the displayed form is plainly positive.  If \(q>0\), set
\[
  A=e\cdot w,\qquad M=m\cdot w.
\]
For fixed \(A\) and \(M\), the smallest possible \(|w|^2\) is
\[
  \frac{A^2}{3}+\frac{M^2}{q}.
\]
Thus positivity is equivalent to positivity of
\[
  a_*A^2+g_*M^2+h_*AM,
\]
where
\[
  a_*=\frac{\theta(b-r)}{3b}+\frac{r^2}{9b},\qquad
  g_*=\frac{\theta}{q}+\frac1b,\qquad
  h_*=\frac{2r/3-\theta}{b}.
\]
A direct simplification gives
\[
  a_*g_*-\frac{h_*^2}{4}
  =\theta\left\{\frac{a_*}{q}
       -\frac{\theta/4-b/3}{b^2}\right\}.
\]
If \(\theta\le4b/3\), the determinant is positive.  Since
\[
  \theta\le \frac{b}{b-3}\le\frac{4b}{3}
\]
whenever \(b\ge15/4\), this proves unconditional positivity in that range.  If
\(\theta>4b/3\), the determinant condition is exactly
\[
  q<\frac{a_*b^2}{\theta/4-b/3}=q_*(b,r).
\]
On a compact entropy-mobility admissible window, the strict gap gives a uniform
lower bound \(c_K\).  Conversely, if \(3<b<15/4\), spectra with sufficiently large
variance and \(\theta>4b/3\) violate the determinant condition, so the symmetric
Jacobian is indefinite.
\end{proof}

\begin{remark}[Relation with the velocity-clock criteria]
\label{rem:3d-heat-entropy-route}
The velocity-clock criteria above do not require this entropy-mobility condition: once a compact window is propagated by the flow-map or trace-gap barriers, Proposition~\ref{prop:3d-compact-window} closes the high-order estimate directly.  Proposition~\ref{prop:3d-fene-mobility} addresses a different question: whether one can build a self-contained three-dimensional logarithmic clock using entropy variables.  In log
variables \(B=\Log C\), the compact FENE set
\[
  K_B(\kappa)=\{B=B^T:\ B\ge(\log\kappa)I,\ \tr(e^B)\le b-\kappa\}
\]
is convex, and heat mollification preserves it.  Thus a heat-relative entropy
energy comparing \(B\) with \(B_\rho=e^{\rho\Delta}B\) remains inside the
physical window.  On entropy-mobility admissible windows the relaxation part is
coercive in the entropy variable; the remaining analytic point is an endpoint
bilinear heat commutator of the form
\[
  \int_0^1 \rho^{-s_0}
  \left|\left\langle B-B_\rho,
  [\calC(B),e^{\rho\Delta}]\nabla u\right\rangle\right|
  \frac{d\rho}{\rho}
  \lesssim
  \norm{\nabla u}_{B^0_{\infty,1}}\norm{B}_{H^{s_0}}^2,
  \qquad \frac32<s_0<2 .
\]
This remark is not used as a theorem below.  Its role is to mark the precise
boundary between the compact-window three-dimensional criteria proved here and
a possible fully self-contained three-dimensional log-conformation theory.
\end{remark}

\section{Hookean limit and relation between the two criteria}
\label{sec:hookean-limit}

The FENE-P criterion reduces to the Oldroyd--B criterion only in a controlled
limit.  Suppose \(0<C\le M\Id\) and \(\tr C\le M\), with \(M\) independent of
\(b\).  Then
\[
  f_b(C)=\frac{b-2}{b-\tr C}
  =1+\frac{\tr C-2}{b}+O_M(b^{-2}),
\]
and hence
\[
  T_b(C)=C-\Id+\frac{\tr C-2}{b}C+O_M(b^{-2}).
\]
Moreover
\[
  \phi_b(C)=-\log b+O_M(b^{-1}),\qquad
  \nabla\phi_b(C)=\frac{\nabla\tr C}{b-\tr C}=O_M(b^{-1})\nabla C .
\]
Thus, on bounded trace windows independent of \(b\), the FENE-P physical
criterion formally reduces to the Oldroyd--B positive-cone criterion.  At fixed
\(b\), or along sequences for which \(\tr C\) approaches \(b\), this reduction
is invalid at the level of pointwise barrier control.  The endpoint theorem
therefore propagates the trace gap before applying the pressure-free high-order
estimate.

\begin{proposition}[Entropy does not control trace-gap oscillation]
\label{prop:barrier-independent-unified}
Fix \(b>2\).  There are smooth diagonal fields \(C_N:\T^2\to\calD_b\) whose
eigenvalues stay in a fixed positive interval and whose positive-cone entropy
\[
  \int_{\T^2}\bigl(\tr C_N-\log\det C_N\bigr)\dd
\]
is uniformly bounded, while
\[
  \norm{\phi_b(C_N)}_{H^{1+\eps}}\to\infty .
\]
\end{proposition}

\begin{proof}
Choose
\[
  \theta_N(x)=\theta_0+a_N\sin(Nx_1),
\]
with \(0<\theta_0<b\), \(a_N\to0\), and
\(a_NN^{1+\eps}\to\infty\).  Let
\[
  C_N(x)=\frac{\theta_N(x)}2\,\Id .
\]
For \(N\) large, the eigenvalues remain in a fixed positive interval and the
positive-cone entropy is uniformly bounded.  However
\[
  \phi_b(C_N)=-\log(b-\theta_N),
\]
whose \(H^{1+\eps}\) norm grows like \(a_NN^{1+\eps}\).  Thus positive-cone
entropy does not control high-frequency oscillation of the trace-gap barrier.
The FENE-P endpoint criterion above avoids assuming this barrier norm directly;
it uses the pointwise barrier equation to recover the compact window and the
pressure-free physical variables to close the derivative estimate.
\end{proof}

\section{Conclusion}

The pressure quotient isolates the part of an elastic stress that can force an
incompressible velocity.  In two dimensions, every smooth spectral isotropic
stress has active part \(q_1Y\) after pressure projection.  The trace-free
conformation equation then supplies the matching principal stretching term
\(2aS(u)\), and the weight \(q_1/a\) cancels the top-order
velocity--stress coupling.  On compact conformation windows this gives the
endpoint high-order coefficient
\[
  1+\norm{\nabla u}_{B^0_{\infty,1}}
  +\norm{\Log C}_{H^{1+\varepsilon}}^2 .
\]

For Oldroyd--B the compact positive-cone window is propagated by the
\(L^1_tB^0_{\infty,1}\) velocity clock, and the pressure-free low-order energy
provides the logarithmic channel required by the high-order estimate.  For
FENE-P the same active-deviatoric cancellation applies after the splitting
\(T_b(C)=(f_ba-1)I+f_bY\); the finite-extensibility trace gap is propagated by
the squared \(L^2_tB^0_{\infty,1}\) clock.

The three-dimensional statements have a narrower scope, and the restriction is
structural.  For spectral isotropic laws with a smooth local Cayley--Hamilton
representation on a compact spectral window, the pressure-free stress splits
into a cancellative \(Y\)-channel and a quadratic residual \((Y^2)^\circ\).  The
residual is generically independent of \(Y\); hence the exact scalar quotient
closure is a two-dimensional phenomenon.  In three dimensions the residual is
estimated separately and absorbed by viscosity on compact windows.  General
anisotropic or non-spectral closures may contain additional pressure-free
directions and are not covered by this quotient closure.  For Oldroyd--B and
FENE-P the quadratic coefficient is zero, and the compact windows used in the
continuation criteria are supplied by the corresponding model-specific barriers.

All criteria are strong-solution continuation criteria.  The proof uses
integer-order Sobolev differentiations, Moser estimates for spectral functions,
and pointwise conformation barriers.  It does not construct Leray-type weak
solutions or prove weak-solution regularization.  This does not leave the
criterion without objects: the local principles in Section~\ref{sec:functional}
give positive-time strong solutions for integer-Sobolev data with compact
initial conformation windows.  The endpoint clocks are therefore blow-up
exclusion mechanisms for maximal strong solutions in that established local
class, not substitutes for a critical \(H^1\) or energy-level well-posedness
theory.  The spectrally admissible class is used only as a compact-window
quotient template.  Thermodynamic polymer laws require the additional
free-energy compatibility and relaxation-dissipation conditions stated in
Remark~\ref{rem:thermodynamic-subclass}.

The fixed-spectrum examples show why the logarithmic channel enters the
estimates: entropy, trace, determinant, and finite-extensibility bounds do not
control high-frequency rotation of the active deviatoric stress.  These examples
are static obstructions for the pressure-free stress map.  They do not prove
that loss of logarithmic regularity forces finite-time blow-up of the evolution.
For the two model systems, the continuation loop closes instead through the
propagated conformation window, the logarithmic regularity estimate, and the
endpoint velocity clock.

\paragraph{Acknowledgements.}
The author acknowledges financial support from the National Natural Science
Foundation of China (NSFC, Grant No. 12501602), the Education Department of
Hunan Province (Grant No. 24C0055), the Science and Technology Department of
Hunan Province (Grant No. 2025JJ60052), and the Scientific Research Start-up
Fund of Xiangtan University (Grant No. KZ0810769).

\end{document}